\documentclass[11pt]{article}

\usepackage{amssymb,amsmath,amsfonts,amsthm}

\usepackage{fancyhdr}

\usepackage{caption2}

\renewcommand\appendix{\par
  \setcounter{section}{0}
   \renewcommand\thesection{Appendix \Alph{section}.}
 }

\newtheorem{theorem}{Theorem}[section]

\newtheorem{lemma}{Lemma}[section]
\newtheorem{proposition}{Proposition}[section]

\newtheorem{remark}{Remark}[section]

\newcommand{\R}{{\mathbb R}}

\newcommand{\T}{{\mathbb T}}
\newcommand{\PP}{{\mathbb P}}

\def\duno{\partial_1}
\def\ddue{\partial_2}
\def\dtre{\partial_3}

\def\dt{\partial_t}
\def\div{{\rm div}\, }

\def\H{\mathcal H }

\def\j{\mathfrak J}
\def\ds{\displaystyle}

\usepackage{lipsum}

\newcommand\blfootnote[1]{%
  \begingroup
  \renewcommand\thefootnote{}\footnote{#1}%
  \addtocounter{footnote}{-1}%
  \endgroup
}

\begin{document}
\title{\bf On well-posedness of the plasma-vacuum interface problem: the case of non-elliptic interface symbol}
\author{{\bf Yuri Trakhinin}\\
Sobolev Institute of Mathematics, Koptyug av. 4, 630090 Novosibirsk, Russia\\
and\\
Novosibirsk State University, Pirogova str. 2, 630090 Novosibirsk, Russia}

\date{
}
%
%
\maketitle
\begin{abstract}
We consider  the plasma-vacuum interface problem in a classical statement when in the plasma region the flow is governed by the equations of ideal compressible magnetohydrodynamics, while in the vacuum region the magnetic field obeys the div-curl system of pre-Maxwell dynamics. The local-in-time existence and uniqueness of the solution to this problem in suitable anisotropic Sobolev spaces was proved in \cite{ST}, provided that at each point of the initial interface the plasma density is strictly positive and the magnetic fields on either side of the interface are not collinear. The non-collinearity condition appears as the requirement that the symbol associated to the interface is elliptic. We now consider the case when this symbol is not elliptic and study the linearized problem, provided that the unperturbed plasma and vacuum non-zero magnetic fields are collinear on the interface. We prove a basic a priori $L^2$ estimate for this problem under the (generalized) Rayleigh-Taylor sign condition $[\partial q/\partial N]<0$ on the jump of the normal derivative of the unperturbed total pressure satisfied at each point of the interface. By constructing an Hadamard-type ill-posedness example for the frozen coefficients linearized problem we show that the simultaneous  failure of the non-collinearity condition and the
Rayleigh-Taylor sign condition leads to Rayleigh-Taylor instability.
\end{abstract}



%
\section{Introduction}
\label{s1}

We consider the equations of ideal compressible magnetohydrodynamics (MHD):
\begin{equation}
\left\{
\begin{array}{l}
\partial_t\rho  +{\rm div}\, (\rho {v} )=0,\\[3pt]
\partial_t(\rho {v} ) +{\rm div}\,(\rho{v}\otimes{v} -{H}\otimes{H} ) +
{\nabla}q=0, \\[3pt]
\partial_t{H} -{\nabla}\times ({v} {\times}{H})=0,\\[3pt]
\partial_t\bigl( \rho e +\frac{1}{2}|{H}|^2\bigr)+
{\rm div}\, \bigl((\rho e +p){v} +{H}{\times}({v}{\times}{H})\bigr)=0,
\end{array}
\right.
\label{1}
\end{equation}
where $\rho$ denotes density, $v\in\mathbb{R}^3$ plasma velocity, $H \in\mathbb{R}^3$ magnetic field, $p=p(\rho,S )$ pressure, $q =p+\frac{1}{2}|{H} |^2$ total pressure, $S$ entropy, $e=E+\frac{1}{2}|{v}|^2$ total energy, and  $E=E(\rho,S )$ internal energy. With a state equation of plasma, $\rho=\rho(p ,S)$, and the first principle of thermodynamics, \eqref{1} is a closed system.

System (\ref{1}) is supplemented by the divergence constraint
\begin{equation}
{\rm div}\, {H} =0
\label{2}
\end{equation}
on the initial data. As is known, taking into account \eqref{2}, we can easily symmetrize system \eqref{1} by rewriting it in the nonconservative form
\begin{equation}
\left\{
\begin{array}{l}
{\displaystyle\frac{\rho_p}{\rho}}\,{\displaystyle\frac{{\rm d} p}{{\rm d}t} +{\rm div}\,{v} =0},\qquad
\rho\, {\displaystyle\frac{{\rm d}v}{{\rm d}t}-({H}\cdot\nabla ){H}+{\nabla}
q  =0 },\\[9pt]
{\displaystyle\frac{{\rm d}{H}}{{\rm d}t} - ({H} \cdot\nabla ){v} +
{H}\,{\rm div}\,{v}=0},\qquad
{\displaystyle\frac{{\rm d} S}{{\rm d} t} =0},
\end{array}\right. \label{3}
\end{equation}
where $\rho_p=\partial\rho/\partial p$ and ${\rm d} /{\rm d} t =\partial_t+({v} \cdot{\nabla} )$. System \eqref{3} is symmetric for the unknown $(p,v,H,S)$. A different symmetrization can be obtained if as the unknown we fix the vector
\[
 U =(q, v,H, S).
\]
Indeed, in terms of $q$ the equation for the pressure in \eqref{3} takes the form
\begin{equation}
\begin{array}{ll}\label{equq}
\displaystyle\frac{\rho_p}{\rho }\left\{\frac{{\rm d} q}{{\rm d}t} -H
\cdot\displaystyle\frac{{\rm d} H}{{\rm d}t} \right\}+{\rm div}\,{v}=0,
\end{array}
\end{equation}
where it is understood that now
$\rho =\rho(q  -|H  |^2/2,S)$ and similarly for $\rho_p $. We then derive ${\rm div}\,{v} $ from \eqref{equq} and rewrite the equation
for the magnetic field in \eqref{3} as
\begin{equation}
\begin{array}{ll}\label{equH}
\displaystyle\frac{{\rm d} H}{{\rm d}t} -(H \cdot\nabla)v  -
 \frac{\rho_p}{\rho }H\left\{\frac{{\rm d} q}{{\rm d}t} -H
\cdot\frac{{\rm d} H}{{\rm d}t} \right\}=0.
\end{array}
\end{equation}
Substituting \eqref{equq}, \eqref{equH} in \eqref{3} then gives the following symmetric system
\begin{equation}
\begin{array}{ll}\label{3'}
\left(\begin{matrix}
{\rho_p/\rho}&\underline
0&-({\rho_p/\rho})H &0 \\
\underline 0^T&\rho
I_3&0_3&\underline 0^T\\
-({\rho_p/\rho})H^T&0_3&I_3+({\rho_p/\rho})H\otimes H&\underline 0^T\\
0&\underline 0&\underline 0&1
\end{matrix}\right)\dt U+\\
 \\
 +
\left( \begin{matrix}
(\rho_p/\rho)
v \cdot\nabla&\nabla\cdot&-({\rho_p/\rho})Hv \cdot\nabla&0\\
\nabla&\rho v \cdot\nabla I_3&-H \cdot\nabla I_3&\underline 0^T\\
-({\rho_p/\rho})H^T v \cdot\nabla&-H \cdot\nabla I_3&(I_3+({\rho_p/\rho})H\otimes H)
v \cdot\nabla&\underline 0^T\\
0&\underline 0&\underline 0&v \cdot\nabla
\end{matrix}\right)
U
=0\,,
\end{array}
\end{equation}
where $\underline 0=(0,0,0)$.
We write system (\ref{3'}) in the form
\begin{equation}
\label{4}
A_0(U )\partial_tU+\sum_{j=1}^3A_j(U )\partial_jU=0,
\end{equation}
which is symmetric hyperbolic provided the hyperbolicity condition $A_0>0$ holds:
\begin{equation}
\rho  >0,\quad \rho_p >0. \label{5}
\end{equation}
Note that $c=1/\sqrt{\rho_p}$ is the sound speed.

Plasma-vacuum interface problems for system \eqref{1} appear in the mathematical modeling of plasma confinement by magnetic fields. This subject is very popular since the 1950's, but most of theoretical studies were devoted to finding stability criteria of equilibrium states. The typical work in this direction is the famous paper of Bernstein et. al. \cite{BFKK} where the plasma-vacuum interface problem was considered in its classical statement modeling the plasma confined inside a perfectly conducting rigid wall and isolated from it by a vacuum region. In this statement (see also, e.g., \cite{Goed}) the plasma is described by the MHD equations \eqref{1} whereas in the vacuum region one considers the so-called {\it pre-Maxwell dynamics}
\begin{equation}
\nabla \times \mathcal{H} =0,\qquad {\rm div}\, \mathcal{H}=0\label{6}
\end{equation}
describing the vacuum magnetic field $\mathcal{H}\in\mathbb{R}^3$. That is, as in the nonrelativistic MHD system \eqref{1}, in the vacuum region after neglecting the displacement current the electric field becomes a secondary variable that may be  computed from the magnetic field.

The classical statement \cite{BFKK,Goed} of the plasma-vacuum problem for systems \eqref{1}/\eqref{4} and \eqref{6} is closed by the boundary conditions
\begin{subequations}
\begin{align}
& \frac{{\rm d}F }{{\rm d} t}=0,\quad [q]=0,\quad  H\cdot N=0 \label{7a}\\
& \mathcal{H}\cdot N=0, \label{7b}
\end{align}
\label{7}
\end{subequations}
on the interface $\Gamma (t)=\{F(t,x)=0\}$ and the initial data
\begin{equation}
\label{8}
\begin{array}{ll}
{U} (0,{x})={U}_0({x}),\quad {x}\in \Omega^{+} (0),\qquad
F(0,{x})=F_0({x}),\quad {x}\in\Gamma(0) , \\
\mathcal{H}(0,x)=
\mathcal{H}_0(x),\quad {x}\in \Omega^{-}(0),
\end{array}
\end{equation}
for the plasma variable $U$, the vacuum magnetic field $\mathcal{H}$ and the function $F$, where $\Omega^+(t)$ and $\Omega ^-(t)$ are space-time domains occupied by the plasma and the vacuum respectively, $N=\nabla F$, and $[q]= q|_{\Gamma}-\frac{1}{2}|\mathcal{H}|^2_{|\Gamma}$ denotes the jump of the total pressure across the interface. The first condition in \eqref{7} means that the interface moves with the velocity of plasma particles at the boundary and since $F$ is an unknown, problem \eqref{4}, \eqref{6}--\eqref{8} is a free-boundary problem. Moreover, in the plasma confinement problem both the plasma and vacuum regions are bounded domains, and at the perfectly conducting rigid wall $\Sigma$ which is the exterior boundary of the vacuum region $\Omega^-(t)$ one states the standard boundary condition $(\mathcal{H}, n)=0$, where $n$ is a normal vector to $\Sigma$.

Until recently, there were no well-posedness results for full (non-stationary) plasma-vacuum models. A basic energy a priori estimate in Sobolev spaces for the linearization of the plasma-vacuum problem \eqref{4}, \eqref{6}--\eqref{8} was first derived in \cite{Tjde}, provided that the  {\it non-collinearity condition}
\begin{equation}
|{H}\times{\mathcal{H}}|_{\Gamma}\geq \delta >0
\label{non-coll}
\end{equation}
($\delta$ is a fixed constant) stating that the magnetic fields on either side of the interface are not collinear holds for a basic state about which one linearizes the problem. The existence of solutions to the linearized problem was then proved in \cite{ST1}.  It was assumed in \cite{ST1,Tjde} that the hyperbolicity conditions \eqref{5} are satisfied in $\Omega^+$ up to the boundary $\Gamma$, i.e. the density does not go to zero continuously, but has a jump (clearly, in the vacuum region $\Omega^-$ the density is identically zero). It is noteworthy that this assumption is compatible with the continuity of the total pressure in \eqref{7a} that allows the pressure $p$ to be positive on the interface for a non-zero $\mathcal{H}|_{\Gamma}$.

In \cite{ST1,Tjde}, for technical simplicity the moving interface  $\Gamma (t)$ was assumed to have the form of a graph $F=x_1-\varphi (t,x')$, $x'=(x_2,x_3)$, i.e. both the plasma and vacuum domains are unbounded. However, as was noted in the subsequent paper \cite{ST}, this assumption is not suitable in a pure form for the original nonlinear free boundary problem \eqref{4}, \eqref{6}--\eqref{8} because in that case the vacuum region $\Omega^-(t)=\{x_1<\varphi (t,x')\}$ is a simply connected domain.  Indeed, the elliptic problem \eqref{6}, \eqref{7b} has then only the trivial solution $\mathcal{H}=0$, and the whole problem is reduced to solving the MHD equations \eqref{1} with a vanishing total pressure $q$ on $\Gamma (t)$.

The technically difficult case of non-simply connected vacuum regions was postponed in \cite{ST} to a future work. Instead of this,
the plasma-vacuum system was assumed in \cite{ST} to be not isolated from the outside world due to a given surface current on the fixed boundary of the vacuum region that forces oscillations. Namely, following \cite{ST}, let us assume that the space domain $\Omega$ occupied by plasma and vacuum is given by
\[
\Omega :=\{x\in\R^3 \; |  \, x_1\in(-1,1),\; x' \in \T^2\}  ,
\]
where $\T^2$ denotes the $2$-torus, which can be thought of as the unit square with periodic boundary conditions. We also set that the moving interface $\Gamma(t) $ is given by
$$
\Gamma(t) := \{ (x_1,x') \in \R \times \T^2  \, , \, x_1=\varphi(t,x')\}, \qquad t \in [0,T]\, ,
$$
where it is assumed that $-1<\varphi(t,\cdot)<1$. Then $\Omega^\pm(t)=\{x_1\gtrless \varphi(t,x')\}\cap\Omega$ are the plasma and vacuum domains respectively. On the fixed {\it top} and {\it bottom} boundaries $$\Gamma_\pm := \{ (\pm 1,x') \, , \, x' \in \T^2 \}$$ of the domain $\Omega$, we prescribe the boundary conditions \cite{ST}
\begin{equation}
\label{9}
v_1=H_1=0 \quad {\rm on } \; [0,T] \times \Gamma_+ \,, \qquad
\nu\times\H=\j \quad {\rm on } \; [0,T] \times \Gamma_- \, ,
\end{equation}
where ${\nu}=(-1,0,0)$ is the outward normal vector at $\Gamma_-$ and  $\j$ represents a given surface current which forces oscillations onto the plasma-vacuum system. In laboratory plasmas this external excitation may be caused by a system of coils. This model can also be exploited for the analysis of waves in astrophysical plasmas, e.g., by mimicking the effects of excitation of MHD waves by an external plasma by means of a localized set of ``coils'', when the response of the internal plasma is the main issue (see a more complete discussion in \cite{Goed}).

Note that with the above parametrization of $\Gamma (t)$, an equivalent formulation of the boundary conditions \eqref{7} at the free interface is
\begin{equation}
\partial_t\varphi =v_N,\quad [q]=0,\quad H_N=0,\quad \mathcal{H}_N=0 \quad \mbox{on}\ \Gamma (t),\label{7'}
\end{equation}
and the initial data read
\begin{equation}
\label{8'}
\begin{array}{ll}
{U} (0,{x})={U}_0({x}),\quad {x}\in \Omega^{+} (0),\qquad
\varphi (0,{x})=\varphi_0({x}),\quad {x}\in\Gamma(0) , \\
\mathcal{H}(0,x)=
\mathcal{H}_0(x),\quad {x}\in \Omega^{-}(0),
\end{array}
\end{equation}
where $v_N=v\cdot N$, $H_N=H\cdot N$, $\mathcal{H}_N=\mathcal{H}\cdot N$, $N=(1,-\partial_2\varphi ,-\partial_3\varphi )$.

Basing on the results of \cite{ST1,Tjde} for the linearized problem, under the non-collinearity condition \eqref{non-coll} satisfied at each point of the initial interface the existence and uniqueness of the solution to the nonlinear plasma-vacuum interface problem \eqref{4}, \eqref{6}, \eqref{9}--\eqref{8'} in suitable anisotropic Sobolev spaces was recently proved in \cite{ST} by a Nash–-Moser-type iteration.

The non-collinearity condition appears in \cite{ST1,ST,Tjde} as the requirement that the symbol associated to the interface is elliptic, i.e. we can resolve the boundary conditions for the gradient $(\partial_t\varphi , \partial_2\varphi , \partial_3\varphi ) $ (see Section \ref{s4} for a detailed discussion). In this paper we study the linearization of problem \eqref{4}, \eqref{6}, \eqref{9}--\eqref{8'} for the case when this symbol is {\it not elliptic}, i.e. when the non-collinearity condition \eqref{non-coll} fails for the basic state (``unperturbed flow") about which we linearize the problem. For this case our main goal is to derive a basic a priori estimate for the linearized problem with variable coefficients under suitable assumptions satisfied for the basic state. We show that the principal assumption is the generalized Rayleigh-Taylor sign condition\footnote{We call it generalized because, unlike the classical Rayleigh-Taylor sign condition, it is for the total pressure $q$ but not for the hydrodynamical pressure $p$. Below we  drop the word ``generalized'' for brevity.}
\begin{equation}
\left[\frac{\partial q}{\partial N}\right]\leq -\epsilon <0
\label{RT}
\end{equation}
on the jump of the normal derivative of the total pressure, where
$$
\left[\frac{\partial q}{\partial N}\right]=\left.\frac{\partial q}{\partial N}\right|_{\Gamma} -\frac{\partial }{\partial N} \left.\left(\frac{1}{2}|\mathcal{H}|^2\right)\right|_{\Gamma}=\left.\frac{\partial q}{\partial N}\right|_{\Gamma}-
\left.\left(\mathcal{H}\cdot\frac{\partial\mathcal{H}}{\partial N}\right)\right|_{\Gamma}
$$
and $\epsilon$ is a fixed constant. Under the fulfilment of condition \eqref{RT} for the basic state we manage to prove a basic a priori $L^2$ estimate for the linearized problem, provided that the unperturbed plasma and vacuum non-zero magnetic fields are collinear on the interface. To be exact, here we mean the straightened interface $\Gamma :=\Omega\cap\{x_1=0\}$
because below, as in \cite{ST1,ST,Tjde}, we reduce the free boundary problem to that in the fixed domains $\Omega^\pm := \Omega \cap \{ x_1\gtrless 0 \}$.

The violation of the non-collinearity condition \eqref{non-coll}  for the basic state means that the unperturbed plasma and vacuum magnetic fields become collinear somewhere on the (straightened) interface (in some regions of $\Gamma$ and/or even in some isolated points of $\Gamma$). That is, our assumption that the unperturbed plasma and vacuum non-zero magnetic fields are collinear at each point of $\Gamma$ does not cover the general case whose consideration is necessary for the usage of a priori estimates obtained for the linearized problem for the proof of existence of solutions to the nonlinear problem. On the other hand, our result is an important and necessary step towards the proof of the well-posedness of the linearized plasma-vacuum problem for which the basic state satisfies the Rayleigh-Taylor sign condition \eqref{RT} (without any restrictions on the relative position of the unperturbed plasma and vacuum magnetic fields on $\Gamma$; see Section \ref{s5} for a further discussion of open problems).

It is interesting to note that in MHD the Rayleigh-Taylor sign condition appears for contact discontinuities in \cite{MTT-cont} in its classical (purely hydrodynamical) form $[\partial p/\partial N]\leq -\epsilon <0$, i.e. as a condition for the pressure $p$ but not for the total pressure $q$ as in \eqref{RT}. As for our problem in the case when the non-collinearity condition \eqref{non-coll} fails, for MHD contact discontinuities the front symbol is also not elliptic and the Rayleigh-Taylor sign condition appears as a condition for the basic state sufficient for the well-posedness of the linearized problem with variable coefficients.

We can also mention two classical examples of free boundary problems with non-elliptic symbol associated with a free boundary. The first one is the free boundary problem for the incompressible  Euler equations with the vacuum boundary condition $p|_{\Gamma}=0$ (see \cite{Lind_incomp} and references therein). The second example is the counterpart of this problem for the compressible Euler equations describing the motion of a compressible perfect liquid (with $\rho|_{\Gamma}>0$) in vacuum \cite{Lind,Tcpam}. The local-in-time existence in Sobolev spaces was managed to be proved in \cite{Lind_incomp,Lind,Tcpam} for the nonlinear free boundary problem only under  the Rayleigh-Taylor sign condition $(\partial p/\partial N)|_{\Gamma}\leq -\epsilon <0$ satisfied at the first moment. As is known, the violation of the Rayleigh-Taylor sign condition for these problems is associated with {\it Rayleigh-Taylor instability} occurring on the level of variable coefficients of the linearized problem. Moreover, for the case of incompressible liquid Ebin \cite{Ebin}
showed the {\it ill-posedness} (in Hadamard's sense) in Sobolev spaces of the nonlinear problem  when the Rayleigh-Taylor sign condition is not satisfied at the first moment.

In this paper we show that Rayleigh-Taylor instability can be easily detected as ill-posedness  already on the level of frozen (constant) coefficients of the linearized problem, where the frozen coefficients linearized problem results from linearization, passage to Alinhac's unknown \cite{Al} (see Section \ref{s2}) and freezing the coefficients at a point of the boundary. We first show this for the above mentioned free boundary problems for the incompressible  and compressible Euler equations. Then, we do the same for our plasma-vacuum problem. More precisely, we prove that one can construct an Hadamard-type ill-posedness example for the frozen coefficients linearized problem {\it if and only if} the non-collinearity condition \eqref{non-coll} and the Rayleigh-Taylor sign condition \eqref{RT} simultaneously fail for frozen coefficients. Note that simple calculations of the corresponding normal modes analysis clarify the physical sense of the non-collinearity condition \eqref{non-coll} appearing in \cite{ST1,ST,Tjde} as the ``mathematical" requirement of the ellipticity of the interface symbol. Roughly speaking, the physical sense is the following: if the plasma and vacuum magnetic fields are not collinear, then they cannot be both collinear with the wave vector (see Remark \ref{r4.1} for more details).

The result of \cite{ST} and that of this paper suggest a hypothesis that for the plasma-vacuum interface problem \eqref{4}, \eqref{6}, \eqref{9}--\eqref{8'} one can prove the local-in-time existence and uniqueness theorem in suitable Sobolev spaces, provided that the Rayleigh-Taylor sign condition \eqref{RT} is satisfied at all those points of the initial interface $\Gamma (0)$ where the non-collinearity condition \eqref{non-coll} fails. In the end of this paper (see Section \ref{s5}) we discuss open problems appearing towards the proof of such an important theorem.

The rest of the paper is organized as follows. In Section \ref{s2}, we reduce the free boundary problem \eqref{4}, \eqref{6}, \eqref{9}--\eqref{8'} to an initial-boundary value problem in fixed domains, obtain the linearized problem and formulate our main results
for it which are Theorem \ref{t2.1} about the basic a priori $L^2$ estimate and Theorem \ref{t2.2} about ill-posedness for frozen coefficients. In Sections \ref{s3} and \ref{s4} we prove Theorems \ref{t2.1} and \ref{t2.2} respectively. At last, in Section \ref{s5} we discuss open problems.

\section{Linearized problem and main results}
\label{s2}

\subsection{Reduced nonlinear problem in fixed domains}

Before linearization we reduce problem \eqref{4}, \eqref{6}, \eqref{9}--\eqref{8'} to that in fixed domains. We straighten the interface $\Gamma (t)$ by using a simple change of independent variables used, for example, in \cite{Met,Tcpam,Tjde}. More precisely, a little modification of this change is necessary to take the presence of the top and bottom boundaries $\Gamma^{\pm}$ into account. The unknowns $U$ and $\mathcal{H}$ being smooth in $\Omega^{\pm}(t)$ are replaced by the vector-functions
\[
\widetilde{U}(t,x ):= {U}(t,{\Phi} (t,x)),\quad
\widetilde{\mathcal{H}}(t,x ):= \mathcal{H}(t,{\Phi} (t,x)).
\]
which are smooth in the fixed domains
\[
\Omega^\pm := \Omega \cap \{ x_1\gtrless 0 \},
\]
and the straightened interface is
\[
\Gamma :=\Omega\cap\{x_1=0\}.
\]
Here
\[
\Phi (t,x ) =(\Phi_1,\Phi_2,\Phi_3 )(t,x ):= ( x_1+\Psi (t,x ),x'),\quad  \Psi(t,x ):= \chi ( x_1)\varphi (t,x'),
\]
$\chi\in C^{\infty}(-1,1)$ equals to 1 on $[-\delta_0,\delta_0]$, where $\delta_0<1$  is a small positive constant, and $\chi (\pm 1 )=0$. Moreover,  if $\|\chi'\|_{L^{\infty}(-1,1)}<\kappa /2$, we assume that $\|\varphi\|_{L^{\infty}([0,T]\times\Gamma )}\leq 1/\kappa $.
This guarantees that the change of variables is admissible because $\partial_1\Phi_1\geq 1/2$.\footnote{If in the local-in-time existence theorem we do not want to assume that the norm $\|\varphi_0\|_{L^{\infty}(\Gamma )}$ of the initial data is small enough (but not just less than 1), we can make a scaling in the definition of $\Phi_1$:  $\Phi_1:=\mu x_1+\Psi$, where $\mu >1$ is large enough.}

Alternatively, we could use the same change of variables as in \cite{ST} inspired by Lannes \cite{Lannes}. This change of variables is the regularization of one half of derivative of the lifting function $\Psi$ with respect to $\varphi$. However, this advantage plays no role for the linearized problem which we are going to study here. On the other hand, the usage of this change of variables was not crucial for the analysis in \cite{ST} and just gave an unimportant gain of one half of derivative in the local-in-time existence theorem for the nonlinear problem \eqref{4}, \eqref{6}, \eqref{9}--\eqref{8'}. Anyway, choosing now the simple change of variables as above, we can postpone the final choice to the nonlinear analysis.

Dropping for convenience tildes in $\widetilde{U}$ and $\widetilde{\mathcal{H}}$, problem \eqref{4}, \eqref{6}, \eqref{9}--\eqref{8'} can be reformulated on the fixed reference domain $\Omega$ as
\begin{equation}
\mathbb{P}(U,\Psi)=0\quad\mbox{in}\ [0,T]\times \Omega^+,\quad \mathbb{V}(\mathcal{H},\Psi)=0\quad\mbox{in}\ [0,T]\times \Omega^-,\label{16}
\end{equation}
\begin{equation}
\mathbb{B}(U,\mathcal{H},\varphi )=\bar\j\quad\mbox{on}\ [0,T] \times(\Gamma^3\times\Gamma_+\times\Gamma_-),\label{17}
\end{equation}
\begin{equation}
(U,\mathcal{H})|_{t=0}=(U_0,\mathcal{H}_0)\quad\mbox{in}\ \Omega^+\times\Omega^-,\qquad \varphi|_{t=0}=\varphi_0\quad \mbox{on}\ \Gamma,\label{18}
\end{equation}
where $\mathbb{P}(U,\Psi)=P(U,\Psi)U$,
\[
P(U,\Psi)=A_0(U)\partial_t +\widetilde{A}_1(U,\Psi)\partial_1+A_2(U )\partial_2+A_3(U )\partial_3,
\]
\[
\widetilde{A}_1(U,\Psi )=\frac{1}{\partial_1\Phi_1}\Bigl(
A_1(U )-A_0(U)\partial_t\Psi -\sum_{k=2}^3A_k(U)\partial_k\Psi \Bigr),
\]
\[
\mathbb{V}(\mathcal{H},\Psi)=\left(
\begin{array}{c}
\nabla\times \mathfrak{H}\\
{\rm div}\,\mathfrak{h}
\end{array}
\right),
\]
\[
\mathfrak{H}=(\mathcal{H}_1\partial_1\Phi_1,\mathcal{H}_{\tau_2},\mathcal{H}_{\tau_3}),\quad
\mathfrak{h}=(\mathcal{H}_{N},\mathcal{H}_2\partial_1\Phi_1,\mathcal{H}_3\partial_1\Phi_1),
\]
\[
\mathcal{H}_{N}=\mathcal{H}_1-\mathcal{H}_2\partial_2\Psi -\mathcal{H}_3\partial_3\Psi,\quad
\mathcal{H}_{\tau_i}=\mathcal{H}_1\partial_i\Psi +\mathcal{H}_i,\quad i=2,3,
\]
\[
\mathbb{B}(U,\mathcal{H},\varphi )=\left(
\begin{array}{c}
\partial_t\varphi -v_{N }\\ {[}q{]} \\ \mathcal{H}_{N }\\v_{1 }  \\ \nu\times\H_{}
\end{array}
\right),\quad [q]=q_{|x_1=0+}-\frac{1}{2}|\mathcal{H}|^2_{x_1=0-}
 ,
\]
\[
v_{N}=v_1- v_2\partial_2\Psi - v_3\partial_3\Psi , \qquad \bar\j=(0,0,0,0,\j)^T.
\]
In \eqref{17} the notation $[0,T] \times(\Gamma^3\times\Gamma_+\times\Gamma_-)$ means that the first three components of this vector equation are taken on $[0,T] \times\Gamma$, the fourth one on $[0,T] \times\Gamma_+$, and the fifth one on $[0,T] \times\Gamma_-$. To avoid an overload of notation we have denoted by the same symbols $v_N,\mathcal{H}_N$ here above and $v_N,\mathcal{H}_N$ as in \eqref{7'}. Notice that
$v_{N |x_1=0}=v_1- v_2\partial_2\varphi - v_3\partial_3\varphi,$ $\mathcal{H}_{N |x_1=0}=
\mathcal{H}_1- \mathcal{H}_2\partial_2\varphi - \mathcal{H}_3\partial_3\varphi $, as in the previous definition in \eqref{7'}.

We did not include in problem \eqref{16}--\eqref{18} the equation
\begin{equation}
{\rm div}\, h=0\quad\mbox{in}\ [0,T]\times \Omega^+,\label{19c}
\end{equation}
and the boundary conditions
\begin{equation}
H_{N}=0\quad\mbox{on}\ [0,T]\times\Gamma, \qquad  H_1=0 \quad {\rm on } \; [0,T] \times \Gamma_+ ,\label{20c}
\end{equation}
where $h=(H_{N},H_2\partial_1\Phi_1,H_3\partial_1\Phi_1)$,
$H_{N}=H_1-H_2\partial_2\Psi -H_3\partial_3\Psi$,
because they are just restrictions on the initial data \eqref{18}. More precisely,  we have the following proposition (see \cite{ST}).

\begin{proposition}
Let the initial data \eqref{18} satisfy \eqref{19c} and \eqref{20c} for $t=0$.
If $(U,\mathcal{H},\varphi )$ is a solution  of problem \eqref{16}--\eqref{18}, then this solution satisfies \eqref{19c} and \eqref{20c} for all $t\in [0,T]$.
\label{p1}
\end{proposition}

\subsection{Basic state}

Let us denote
\begin{equation}
\begin{array}{ll}\label{defQ}
Q_T:= (-\infty,T]\times\Omega,\quad Q^\pm_T:= (-\infty,T]\times\Omega^\pm,\\
 \omega_T:=(-\infty,T]\times\Gamma, \quad
\omega^\pm_T:=(-\infty,T]\times\Gamma_\pm.
\end{array}
\end{equation}
Let
\begin{equation}
(\widehat{U}(t,x ),\widehat{\mathcal{H}}(t,x ),\hat{\varphi}(t,{x}'))
\label{21}
\end{equation}
be a given sufficiently smooth vector-function with $\widehat{U}=(\hat{q},\hat{v},\widehat{H},\widehat{S})$ defined on $Q^+_T\times Q^-_T\times\omega_T$, with
\begin{equation}
\begin{array}{ll}\label{22}
\|\widehat{U}\|_{W^{2,\infty}(Q^+_T)}  + \|\widehat{\mathcal{H}}\|_{W^{2,\infty}(Q^+_T)} + \|\hat{\varphi}\|_{W^{2,\infty}(\omega_T)} \le K,
\qquad
 \| \hat\varphi\|_{L^{\infty}([0,T]\times\Gamma )}\leq 1/\kappa,
\end{array}
\end{equation}
where $K>0$ is a constant and $\kappa$ is the positive constant introduced above (we recall that $\|\chi'\|_{L^{\infty}(-1,1)}<\kappa /2$). Let also
\[
\widehat{\Phi} (t,x ) =(\widehat{\Phi}_1,\widehat{\Phi}_2,\widehat{\Phi}_3 )(t,x ):= ( x_1+\widehat{\Psi} (t,x ),x'),\quad  \widehat{\Psi}(t,x ):= \chi ( x_1)\hat{\varphi} (t,x').
\]

We assume that the basic state \eqref{21} satisfies
(for some positive $\rho_0,\rho_1\in\R$)
\begin{equation}
\rho (\hat{p},\widehat{S})\geq \rho_0 >0,\quad \rho_p(\hat{p},\widehat{S})\ge \rho_1 >0 \qquad \mbox{in}\ \overline{Q}^+_T,
\label{23}
\end{equation}
\begin{equation}
\partial_t\widehat{H}+\frac{1}{\partial_1\widehat{\Phi}_1}\left\{ (\hat{w} \cdot\nabla )
\widehat{H} - (\hat{h} \cdot\nabla ) \hat{v} + \hat{H}{\rm div}\,\hat{u}\right\} =0\qquad \mbox{in}\ \overline{Q}^+_T,
\label{26}
\end{equation}
\begin{equation}
 {\rm div}\,\hat{\mathfrak{h}}=0\qquad \mbox{in}\ Q^-_T,
\label{25}
\end{equation}
\begin{equation}
\partial_t\hat{\varphi}-\hat{v}_{N}=0,\quad \hat{\mathcal{H}}_N=0 \quad \mbox{on}\,\; \omega_T,\quad \hat{{v}}_1=0 \quad \mbox{on}\,\; \omega_T^+,\quad \nu\times\hat{\mathcal{H}}=\j \quad \mbox{on}\,\; \omega_T^-,\label{24}
\end{equation}
where all the ``hat'' values are determined like corresponding values for $(U,\mathcal{H},\varphi)$, i.e.
\[
\widehat{\mathfrak{H}}=(\widehat{\mathcal{H}}_1\partial_1\widehat{\Phi}_1,
\widehat{\mathcal{H}}_{\tau_2},\widehat{\mathcal{H}}_{\tau_3}),\quad
\hat{\mathfrak{h}}=(\widehat{\mathcal{H}}_{N},\widehat{\mathcal{H}}_2\partial_1\widehat{\Phi}_1,
\widehat{\mathcal{H}}_3\partial_1\widehat{\Phi}_1),
\quad \hat{h}=(\widehat{H}_{N},\widehat{H}_2\partial_1\widehat{\Phi}_1,\widehat{H}_3\partial_1\widehat{\Phi}_1),
\]
\[
 \hat{{H}}_{N}=\widehat{{H}}_1-
\hat{{H}}_2\partial_2\widehat{\Psi}- \hat{{H}}_3\partial_3\widehat{\Psi},\quad
\widehat{\mathcal{H}}_{N}=\widehat{\mathcal{H}}_1- \widehat{\mathcal{H}}_2\partial_2\widehat{\Psi}- \widehat{\mathcal{H}}_3\partial_3\widehat{\Psi},
\]
\[
\hat p=\hat q  -|\widehat H  |^2/2 ,\quad
\hat{v}_{N}=\hat{v}_1- \hat{v}_2\partial_2\widehat{\Psi}- \hat{v}_3\partial_3\widehat{\Psi},\quad
\]
\[
\hat{u}=(\hat{v}_{N},\hat{v}_2\partial_1\widehat{\Phi}_1,\hat{v}_3\partial_1\widehat{\Phi}_1),\quad
\hat{w}=\hat{u}-(\partial_t\widehat{\Psi},0,0).
\]
It follows from (\ref{26}) that the constraints
\begin{equation}
{\rm div}\,\hat{h}=0\quad \mbox{in}\; Q^+_T,\quad \hat{H}_{N}=0\quad \mbox{on}\,\; \omega_T,
\quad  \hat{H}_1=0 \quad {\rm on } \; \omega_T^+, \label{27}
\end{equation}
are satisfied for the basic state (\ref{21}) if they hold at $t=0$ (see \cite{T09} for the proof).
Thus, for the basic state we also require the fulfillment of conditions
\eqref{27} at $t=0$.

\subsection{Linearized problem}

The linearized equations for (\ref{16}), (\ref{17}) read:
\[
\mathbb{P}'(\widehat{U},\widehat{\Psi})(\delta U,\delta\Psi):=
\frac{\rm d}{{\rm d}\varepsilon}\mathbb{P}(U_{\varepsilon},\Psi_{\varepsilon})|_{\varepsilon =0}=f
\qquad \mbox{in}\ Q^+_T,
\]
\[
\mathbb{V}'(\widehat{\mathcal{H}},\widehat{\Psi})(\delta \mathcal{H},\delta\Psi):=
\frac{\rm d}{{\rm d}\varepsilon}\mathbb{V}(\mathcal{H}_{\varepsilon},\Psi_{\varepsilon})|_{\varepsilon =0}=\mathcal{G}'
\qquad \mbox{in}\ Q^-_T,
\]
\[
\mathbb{B}'(\widehat{U},\widehat{\mathcal{H}},\hat{\varphi})(\delta U,\delta \mathcal{H},\delta \varphi ):=
\frac{\rm d}{{\rm d}\varepsilon}\mathbb{B}(U_{\varepsilon},\mathcal{H}_{\varepsilon},\varphi_{\varepsilon})|_{\varepsilon =0}={g}
\qquad \mbox{on}\ \omega_T^3\times\omega_T^\pm,
\]
where $U_{\varepsilon}=\widehat{U}+ \varepsilon\,\delta U$, $\mathcal{H}_{\varepsilon}=
\widehat{\mathcal{H}}+\varepsilon\,\delta \mathcal{H}$,
$\varphi_{\varepsilon}=\hat{\varphi}+ \varepsilon\,\delta \varphi$, $\Psi_{\varepsilon}=\chi\varphi_{\varepsilon}$ and
$\delta\Psi=\chi\delta\varphi$. In the above boundary equation the first three components are taken on $\omega_T$, the fourth one on $\omega_T^+$, and the fifth one on $\omega_T^-$.
Here we introduce the source terms $f=(f_1,\ldots ,f_8)$, $\mathcal{G}'=(\chi,
\Xi)$, $\chi=(\chi_1, \chi_2,
\chi_3)$, and $g=(g_1,g_2,g_3)$ to make the interior equations and the boundary conditions inhomogeneous.

The exact form of the linearized equations (below we drop $\delta$) read:
\[
\mathbb{P}'(\widehat{U},\widehat{\Psi})(U,\Psi)
=
P(\widehat{U},\widehat{\Psi})U +{\mathcal C}(\widehat{U},\widehat{\Psi})
U -   \bigl\{P(\widehat{U},\widehat{\Psi})\Psi\bigr\}\frac{\partial_1\widehat{U}}{\partial_1\widehat{\Phi}_1}
=f,
\]
\[
\mathbb{V}'(\widehat{\mathcal{H}},\widehat{\Psi})(\mathcal{H},\Psi)=
\mathbb{V}(\mathcal{H},\widehat{\Psi})+
\left(\begin{array}{c}
\nabla\widehat{\mathcal{H}}_1\times\nabla\Psi\\[3pt]
\nabla \times \left(\begin{array}{c} 0 \\ -\widehat{\mathcal{H}}_3 \\
\widehat{\mathcal{H}}_2 \end{array} \right) \cdot \nabla\Psi
\end{array}
\right)=\mathcal{G}',
\]
\[
\mathbb{B}'(\widehat{U},\widehat{\mathcal{H}},\hat{\varphi})(U,\mathcal{H},\varphi )=
\left(
\begin{array}{c}
\partial_t\varphi +\hat{v}_2\partial_2\varphi+\hat{v}_3\partial_3\varphi -v_{N}\\[3pt]
q-\widehat{\mathcal{H}} \cdot \mathcal{H}\\[3pt]
\mathcal{H}_N-\widehat{\mathcal{H}}_2\partial_2\varphi -\widehat{\mathcal{H}}_3\partial_3\varphi
\\v_1
\\\nu\times\H
\end{array}
\right)=g,
\]
where $v_{N}:= v_1-v_2\partial_2\widehat{\Psi}-v_3\partial_3\widehat{\Psi}$ and the matrix
${\mathcal C}(\widehat{U},\widehat{\Psi})$ is determined as follows:
\[
\begin{array}{r}
{\mathcal C}(\widehat{U},\widehat{\Psi})Y
= (Y ,\nabla_yA_0(\widehat{U} ))\partial_t\widehat{U}
 +(Y ,\nabla_y\widetilde{A}_1(\widehat{U},\widehat{\Psi}))\partial_1\widehat{U}
 \\[6pt]
+ (Y ,\nabla_yA_2(\widehat{U} ))\partial_2\widehat{U}
+ (Y ,\nabla_yA_3(\widehat{U} ))\partial_3\widehat{U},
\end{array}
\]
\[
(Y ,\nabla_y A(V)):=\sum_{i=1}^8y_i\left.\left(\frac{\partial A (Y )}{
\partial y_i}\right|_{Y =V}\right),\quad Y =(y_1,\ldots ,y_8).
\]
Since the differential operators $\mathbb{P}'(\widehat{U},\widehat{\Psi})$ and $\mathbb{V}'(\widehat{\mathcal{H}},\widehat{\Psi})$ are first-order operators in $\Psi$, we rewrite the linearized problem in terms of the Alinhac's ``good unknown'' (see \cite{Al})
\begin{equation}
\dot{U}:=U -\frac{\Psi}{\partial_1\widehat{\Phi}_1}\,\partial_1\widehat{U},\quad
\dot{\mathcal{H}}:=\mathcal{H} -\frac{\Psi}{\partial_1\widehat{\Phi}_1}\,\partial_1\widehat{\mathcal{H}}.
\label{29}
\end{equation}
Taking into account assumptions \eqref{24} and omitting detailed calculations,
we rewrite our linearized equations in terms of the new unknowns \eqref{29}:
\begin{equation}\label{30}
\begin{array}{ll}
\ds \PP'(\widehat{U},\widehat{\Psi})({U},{\Psi})=P(\widehat{U},\widehat{\Psi})\dot{U} +{\mathcal C}(\widehat{U},\widehat{\Psi})
\dot{U} + \frac{\Psi}{\partial_1\widehat{\Phi}_1}\,\partial_1\bigl\{\mathbb{P}
(\widehat{U},\widehat{\Psi})\bigr\}=f,\\
\\
\ds \mathbb{V}'(\widehat{\mathcal{H}},\widehat{\Psi})({\mathcal{H}},{\Psi})=\mathbb{V}(\dot{\mathcal{H}},\widehat{\Psi})+ \frac{\Psi}{\partial_1\widehat{\Phi}_1}\,\partial_1\bigl\{\mathbb{V}
(\widehat{\mathcal{H}},\widehat{\Psi})\bigr\}=\mathcal{G}',
\end{array}
\end{equation}
\begin{multline}
\mathbb{B}'_e(\widehat{U},\widehat{\mathcal{H}},\hat{\varphi})(\dot{U},\dot{\mathcal{H}},\varphi ):= \mathbb{B}'(\widehat{U},\widehat{\mathcal{H}},\hat{\varphi})(U,\mathcal{H},\varphi )\\[6pt] =
 \left(
\begin{array}{c}
\partial_t\varphi+\hat{v}_2\partial_2\varphi+\hat{v}_3\partial_3\varphi-\dot{v}_{N}-
\varphi\,\partial_1\hat{v}_{N}\\[3pt]
\dot{q}-\widehat{\mathcal{H}} \cdot \dot{\mathcal{H}}+ [\partial_1\hat{q}]\varphi \\[3pt]
\dot{\mathcal{H}}_{N}-\partial_2\bigl(\widehat{\mathcal{H}}_2\varphi \bigr) -\partial_3\bigl(\widehat{\mathcal{H}}_3\varphi \bigr)\\[3pt]
\dot{v}_1\\[3pt]
\nu\times\dot\H
\end{array}\right)=g,\qquad
\label{32}
\end{multline}
where $\dot{v}_{\rm N}=\dot{v}_1-\dot{v}_2\partial_2\hat{\Psi}-\dot{v}_3\partial_3\hat{\Psi}$,
$\dot{\mathcal{H}}_{N}=\dot{\mathcal{H}}_1-\dot{\mathcal{H}}_2\partial_2\hat{\Psi}-\dot{\mathcal{H}}_3\partial_3\hat{\Psi}$, and
\[
[\partial_1\hat{q}]=(\partial_1\hat{q})|_{x_1=0}-(\widehat{\mathcal{H}} \cdot \partial_1\widehat{\mathcal{H}})|_{x_1=0}.
\]
We used assumption \eqref{25}, taken at $x_1=0$, while writing down the third boundary condition in \eqref{32}.

As in \cite{Al} (and, e.g., in \cite{ST,T09,Tcpam}), we drop the zeroth-order terms in $\Psi$ in \eqref{30} and consider the effective linear operators
\begin{equation}
\begin{array}{ll}\label{effectivelineq}

\mathbb{P}'_e(\widehat{U},\widehat{\Psi})\dot{U} :=P(\widehat{U},\widehat{\Psi})\dot{U} +{\mathcal C}(\widehat{U},\widehat{\Psi})
\dot{U}=f, \\
\ds \mathbb{V}'_e(\widehat{\mathcal{H}},\widehat{\Psi})\dot{\mathcal{H}}:=\mathbb{V}(\dot{\mathcal{H}},\widehat{\Psi})=\mathcal{G}'.
\end{array}
\end{equation}
In the future nonlinear analysis the dropped terms in (\ref{30}) should be considered as error terms (see \cite{Al,ST,T09,Tcpam}).
With the new form \eqref{effectivelineq}, \eqref{32} of the linearized equations, our linearized problem for $(\dot{U},\dot{\mathcal{H}},\varphi )$ reads in explicit form:
\begin{subequations}\label{34}
\begin{align}
\widehat{A}_0\partial_t\dot{U}+\sum_{j=1}^{3}\widehat{A}_j\partial_j\dot{U}+
\widehat{\mathcal C}\dot{U}=f \qquad &\mbox{in}\ Q^+_T,\label{34a}
\\
\nabla\times \dot{\mathfrak{H}}=\chi,\quad {\rm div}\,\dot{\mathfrak{h}}=\Xi \qquad &\mbox{in}\ Q^-_T, \label{36'}
\\
\partial_t\varphi=\dot{v}_{N}-\hat{v}_2\partial_2\varphi-\hat{v}_3\partial_3\varphi +
\varphi\,\partial_1\hat{v}_{N}+g_1,  \qquad &\label{34b}
\\
\dot{q}=\widehat{\mathcal{H}}\cdot\dot{\mathcal{H}}-  [ \partial_1\hat{q}] \varphi +g_2, \qquad & \label{35}
\\
\dot{\mathcal{H}}_{N} =\partial_2\bigl(\widehat{\mathcal{H}}_2\varphi \bigr) +\partial_3\bigl(\widehat{\mathcal{H}}_3\varphi \bigr)+g_3\qquad &\mbox{on}\ \omega_T,
\label{37}
\\
\dot{{v}}_1=g_4 \quad \mbox{on}\,\; \omega_T^+,\qquad \nu\times\dot{\mathcal{H}}=g_5 \quad &\mbox{on}\,\; \omega_T^-,\label{37a}
\\
(\dot{U},\dot{\mathcal{H}},\varphi )=0\qquad &\mbox{for}\ t<0,\label{38a}
\end{align}
\end{subequations}
where
\[
\widehat{A}_{\alpha}=:{A}_{\alpha}(\widehat{U} ),\quad \alpha =0,2,3,\quad
\widehat{A}_1=:\widetilde{A}_1(\widehat{U} ,\widehat{\Psi}),\quad
\widehat{\mathcal C}:={\mathcal C}(\widehat{U},\widehat{\Psi}),
\]
\[
\dot{\mathfrak{H}}=(\dot{\mathcal{H}}_1\partial_1\widehat{\Phi}_1,\dot{\mathcal{H}}_{\tau_2},\dot{\mathcal{H}}_{\tau_3}),\quad
\dot{\mathfrak{h}}=(\dot{\mathcal{H}}_{N},\dot{\mathcal{H}}_2\partial_1\widehat{\Phi}_1,\dot{\mathcal{H}}_3\partial_1\widehat{\Phi}_1),
\]
\[
\dot{\mathcal{H}}_{N}=\dot{\mathcal{H}}_1-\dot{\mathcal{H}}_2\partial_2\widehat{\Psi}-\dot{\mathcal{H}}_3\partial_3\widehat{\Psi},\quad
\dot{\mathcal{H}}_{\tau_i}=\dot{\mathcal{H}}_1\partial_i\widehat{\Psi}+\dot{\mathcal{H}}_i,\quad i=2,3.
\]
For the resolution of the elliptic problem \eqref{36'}, \eqref{37}, \eqref{37a} the data $\chi,g_5$ must satisfy some necessary compatibility conditions described in \cite{ST}.

We assume that the source terms $f, \chi,\Xi $ and the boundary data $g$ vanish in the past and consider the case of zero initial data. We postpone the case of nonzero initial data to the nonlinear analysis (see e.g. \cite{ST,Tcpam}).

\subsection{Reduced linearized problem}

From problem \eqref{34} we can deduce nonhomogeneous counterparts of the divergence constraint ${\rm div}\,\dot h=0$ and the ``redundant'' boundary conditions $\dot H_N|_{x_1=0}=0$ and $\dot H_1|_{x_1=1}=0$ (see \cite{ST}), where $\dot{h}=(\dot{H}_{N},\dot{H}_2\partial_1\widehat{\Phi}_1,\dot{H}_3\partial_1\widehat{\Phi}_1)$, $\dot{H}_{N}=\dot{H}_1-\dot{H}_2\partial_2\widehat{\Psi}-\dot{H}_3
\partial_3\widehat{\Psi}.$ Following \cite{ST}, we  can use an appropriate shift of the unknown
\[
U^{\natural}=\dot{U}-\widetilde{U},\quad \mathcal{H}^{\natural}=\dot{\mathcal{H}}-{\mathcal{H}}''
\]
to make homogeneous not only the mentioned divergence constraint and the ``redundant" boundary conditions but also all the boundary conditions \eqref{34b}--\eqref{37a} as well as the elliptic system \eqref{36'}, where the construction of the shifting functions $\widetilde{U}$ and ${\mathcal{H}}''$ is described in \cite{ST} in full details.

Dropping for convenience the indices $^{\natural}$ above, our reduced linearized problem in terms of the new unknown $(U,\mathcal{H}):=(U^{\natural},\mathcal{H}^{\natural})$ reads
\begin{subequations}\label{34'}
\begin{align}
\widehat{A}_0\partial_t{U}+\sum_{j=1}^{3}\widehat{A}_j\partial_j{U}+
\widehat{\mathcal C}{U}=F \qquad &\mbox{in}\ Q^+_T,\label{34'a}
\\
\nabla\times {\mathfrak{H}}=0,\quad {\rm div}\,{\mathfrak{h}}=0 \qquad &\mbox{in}\ Q^-_T, \label{36'b}
\\
\partial_t\varphi={v}_{N}-\hat{v}_2\partial_2\varphi-\hat{v}_3\partial_3\varphi +
\varphi\,\partial_1\hat{v}_{N},  \qquad &
\\
{q}=\widehat{\mathcal{H}}\cdot{\mathcal{H}}-  [ \partial_1\hat{q}] \varphi , \qquad & \label{35'd}
\\
{\mathcal{H}}_{N} =\partial_2\bigl(\widehat{\mathcal{H}}_2\varphi \bigr) +\partial_3\bigl({\widehat{H}}_3\varphi \bigr)\qquad &\mbox{on}\ \omega_T,
\label{37'}
\\
{{v}}_1=0 \quad \mbox{on}\,\; \omega_T^+,\qquad \nu\times{\mathcal{H}}=0 \quad &\mbox{on}\,\; \omega_T^-,\label{37'a}
\\
({U},{\mathcal{H}},\varphi )=0\qquad & \mbox{for}\ t<0.\label{38'f}
\end{align}
\end{subequations}
and solutions should satisfy
\begin{equation}
{\rm div}\,{h}=0\qquad\mbox{in}\ Q^+_T,
\label{93}
\end{equation}
\begin{equation}
{H}_{N}=\widehat{H}_2\partial_2\varphi +\widehat{H}_3\partial_3\varphi -
\varphi\,\partial_1\widehat{H}_{N}\quad\mbox{on}\ \omega_T, \qquad {H}_{1}=0\quad\mbox{on}\ \omega_T^+.
\label{95}
\end{equation}
All the notations here for $U$ and $\mathcal{H}$ (e.g., $h$, $\mathfrak{H}$, $\mathfrak{h}$, etc.) are analogous to the corresponding ones for $\dot{U}$ and $\dot{\mathcal{H}}$ introduced above. Note also that, in view of \eqref{93}, the first condition in \eqref{95} can be rewritten as
\begin{equation}
{H}_{N}=\partial_2(\widehat{H}_2\varphi )+\partial_3(\widehat{H}_3\varphi )\quad\mbox{on}\ \omega_T^+.
\label{95'}
\end{equation}

Again, as in \cite{ST}, it will be convenient for us to make use of different ``plasma" variables and an equivalent form of equations \eqref{34'a}. We define the matrix
\begin{equation}
\begin{array}{ll}\label{defeta}
\hat \eta=\begin{pmatrix}
 1&-\ddue\widehat\Psi &-\dtre\widehat\Psi \\
0 &\duno\widehat\Phi_1&0\\
0&0&\duno\widehat\Phi_1
\end{pmatrix}.
\end{array}
\end{equation}
It follows that
\begin{equation}
\begin{array}{ll}\label{defcalU}
{u}=({v}_{N},{v}_2\partial_1\widehat{\Phi}_1,{v}_3\partial_1\widehat{\Phi}_1)=\hat \eta\, v, \qquad
{h}=({H}_{N},{H}_2\partial_1\widehat{\Phi}_1,{H}_3\partial_1\widehat{\Phi}_1)=\hat\eta \,H.
\end{array}
\end{equation}
Multiplying \eqref{34'a} on the left side by the matrix
\begin{equation*}
\begin{array}{ll}\label{}
\widehat R=\begin{pmatrix}
 1&\underline 0&\underline 0&0 \\
 \underline0^T&\hat \eta&0_3& \underline 0^T\\
 \underline 0^T&0_3&\hat \eta&\underline 0^T\\
 0&\underline 0^T&\underline 0^T&1
\end{pmatrix},
\end{array}
\end{equation*}
after some calculations we get the symmetric hyperbolic system for the new vector of unknowns $\mathcal{U}=(q,u,h,S)$
(compare with \eqref{3'}, \eqref{34'a}):
\begin{equation}
\begin{array}{ll}\label{34'''}
\duno\widehat\Phi_1
\left(\begin{matrix}
{\hat{\rho}_p/\hat{\rho}}&\underline
0&-({\hat{\rho}_p/\hat{\rho}})\hat{h} &0 \\
\underline 0^T&\hat{\rho}
\hat a_0&0_3&\underline 0^T\\
-({\hat{\rho}_p/\hat{\rho}})\hat{h}^T&0_3&\hat a_0 +({\hat{\rho}_p/\hat{\rho}})\hat{h}\otimes\hat{h}&\underline 0^T\\
0&\underline 0&\underline 0&1
\end{matrix}\right)\dt
\left(\begin{matrix}
q \\ u \\
h\\S \end{matrix}\right)
 +
\left( \begin{matrix}
0&\nabla\cdot&\underline 0&0\\
\nabla&0_3&0_3 &\underline 0^T\\
\underline 0^T&0_3 &0_3&\underline 0^T\\
0&\underline 0&\underline 0&0
\end{matrix}\right)
\left(\begin{matrix}q \\ u \\ h\\S \end{matrix}\right)
\\
 \\
 +
\duno\widehat\Phi_1
\left( \begin{matrix}
(\hat{\rho}_p/\hat{\rho})
\hat w \cdot\nabla&\nabla\cdot&-({\hat{\rho}_p/\hat{\rho}})\hat{h}\hat w \cdot\nabla&0\\
\nabla&\hat{\rho} \hat a_0\hat w \cdot\nabla&-\hat a_0\hat{h} \cdot\nabla &\underline 0^T\\
-({\hat{\rho}_p/\hat{\rho}})\hat{h}^T \hat w \cdot\nabla&-\hat a_0\hat{h} \cdot\nabla &(\hat a_0 +({\hat{\rho}_p/\hat{\rho}})\hat{h}\otimes\hat{h})
\hat w \cdot\nabla&\underline 0^T\\
0&\underline 0&\underline 0&\hat w \cdot\nabla
\end{matrix}\right)
\left(\begin{matrix}q \\ u \\ h\\S \end{matrix}\right)
+\widehat{\mathcal{C}}'\mathcal{U}=\mathcal{F}\,,
\end{array}
\end{equation}
where $\hat{\rho}:=\rho (\hat{p},\widehat{S})$, $\hat{\rho}_p:=\rho_p (\hat{p},\widehat{S})$, and $\hat a_0$ is the symmetric and positive definite matrix
$$
\hat a_0 =(\hat \eta^{-1})^T\hat \eta^{-1},$$
with a new matrix $\widehat{\mathcal{C}}'$ in the zero-order term (whose precise form has no importance) and where we have set
$
\mathcal{F}=\duno\widehat\Phi_1 \, \widehat R
F.$
We write system \eqref{34'''} in compact form as
\begin{equation}
\begin{array}{ll}
\displaystyle \widehat{\mathcal{A}}_0\partial_t{\mathcal{U}}+\sum_{j=1}^{3}(\widehat{\mathcal{A}}_j+{\mathcal{E}}_{1j+1})\partial_j{\mathcal{U}}+
\widehat{\mathcal C}'{\mathcal{U}}=\mathcal{F} ,\label{73}
\end{array}
\end{equation}
where
\[
\mathcal{E}_{12}=\left(\begin{array}{cccccc}
0& 1 &0 &0 & \cdots & 0 \\
1 & 0 &0 &0 & \cdots & 0 \\
0 & 0 &0 &0 & \cdots & 0 \\
0 & 0 &0 &0 & \cdots & 0 \\
\vdots & \vdots &\vdots &\vdots& & \vdots \\
0& 0 &0 &0 & \cdots & 0  \end{array}
 \right), \qquad
\,
\mathcal{E}_{13}=\left(\begin{array}{cccccc}
0& 0 &1 &0 & \cdots & 0 \\
0 & 0 &0 &0 & \cdots & 0 \\
1 & 0 &0 &0 & \cdots & 0 \\
0 & 0 &0 &0 & \cdots & 0 \\
\vdots & \vdots &\vdots &\vdots& & \vdots \\
0& 0 &0 &0 & \cdots & 0  \end{array}
 \right),
 \]
\[
\mathcal{E}_{14}=\left(\begin{array}{cccccc}
0& 0 &0 & 1&\cdots & 0 \\
0 & 0 &0 &0& \cdots & 0 \\
0 & 0 &0 &0& \cdots & 0 \\
1 & 0 &0 &0& \cdots & 0 \\
\vdots &\vdots & \vdots &\vdots & & \vdots \\
0& 0 &0 &0& \cdots & 0
\end{array}
 \right).
\]
The formulation \eqref{73} has the advantage of the form of the boundary matrix of the system $\widehat{\mathcal{A}}_1+{\mathcal{E}}_{12}$, with
\begin{equation}
\begin{array}{ll}\label{a10}
\widehat{\mathcal{A}}_1=0 \qquad\mbox{on }\omega_T\cup\omega_T^+,
\end{array}
\end{equation}
because $\hat w_1=\hat h_1=0$, and ${\mathcal{E}}_{12}$ is a constant matrix.
That is, system \eqref{73} is symmetric hyperbolic with characteristic boundary of constant multiplicity \cite{Rauch}. The final form of our reduced linearized problem is thus
\begin{subequations}\label{34'new}
\begin{alignat}{2}
\displaystyle
& \widehat{\mathcal{A}}_0\partial_t{\mathcal{U}}+\sum_{j=1}^{3}(\widehat{\mathcal{A}}_j+{\mathcal{E}}_{1j+1})\partial_j{\mathcal{U}}+
\widehat{\mathcal C}'{\mathcal{U}}=\mathcal{F}
  &\qquad\mbox{in}&\ Q^+_T,\label{34'anew}
\\
&  \nabla\times {\mathfrak{H}}=0,\quad {\rm div}\,{\mathfrak{h}}=0 &\qquad\mbox{in}&\ Q^-_T, \label{36'bnew}
\\
&  \partial_t\varphi=v_N-\hat{v}_2\partial_2\varphi-\hat{v}_3\partial_3\varphi +
\varphi\,\partial_1\hat{v}_{N},   & & \label{35"}
\\
&  {q}=\widehat{\mathcal{H}}\cdot{\mathcal{H}}-  [ \partial_1\hat{q}] \varphi ,  & &\label{35'dnew}
\\
&  {\mathcal{H}}_{N} =\partial_2\bigl(\widehat{\mathcal{H}}_2\varphi \bigr) +\partial_3\bigl(\widehat{\mathcal{H}}_3\varphi \bigr) &\qquad\mbox{on}&\ \omega_T,
\label{37'new}
\\
&  {{v}}_1=0 &\qquad \mbox{on}&\ \omega_T^+,\label{37''anew}
\\
& \mathcal{H}_2=\mathcal{H}_3=0  &\qquad\mbox{on}&\ \omega_T^-,\label{37'anew}
\\
&  (\mathcal{U},{\mathcal{H}},\varphi )=0 & \qquad\mbox{for}&\ t<0,\label{38'fnew}
\end{alignat}
\end{subequations}
Recall that the solutions of problem \eqref{34'new} satisfy \eqref{93} and
\eqref{95}.

\subsection{Main results}

We are now in a position to state our main results.

\begin{theorem}
Let the basic state \eqref{21} satisfies assumptions \eqref{23}--\eqref{27}, the collinearity condition
\begin{equation}
\widehat{\mathcal{H}}|_{x_1=0}=\hat{\beta}\widehat{H}|_{x_1=0}\qquad \mbox{on}\ \omega_T\label{collin}
\end{equation}
and the Rayleigh-Taylor sign condition\footnote{Inequality \eqref{39} is just the Rayleigh-Taylor sign condition \eqref{RT} written for the straightened unperturbed interface (with the equation $x_1=0$).}
\begin{equation}
[\partial_1\hat{q}] \geq \epsilon >0\qquad \mbox{on}\ \omega_T,\label{39}
\end{equation}
where $\hat{\beta}=\hat{\beta} (t,x')$ is a smooth function, $[\partial_1\hat{q}]=\partial_1\hat{q}|_{x_1=0}-(\widehat{\mathcal{H}}\cdot\partial_1\widehat{\mathcal{H}})|_{x_1=0}$ and $\epsilon$ is a fixed constant. Then, sufficiently smooth solutions $(\mathcal{U},{\mathcal{H}},\varphi )$ of problem \eqref{34'new} obey the a priori estimate
\begin{equation}
\|\mathcal{U}\|_{L^2(Q_T^+)}+\|{\mathcal{H}}\|_{L^2(Q_T^-)}+\|\varphi\|_{L^2(\omega_T)}\leq C\|\mathcal{F}  \|_{L^2(Q_T^+)},
\label{est}
\end{equation}
where $C=C(K,T,\epsilon )>0$ is a constant independent of the data
$\mathcal{F}$.
\label{t2.1}
\end{theorem}

Here we have formulated the result for the reduced problem \eqref{34'new}. Returning to the original linearized problem \eqref{34} with inhomogeneous boundary conditions and inhomogeneous equations in the vacuum region, as in \cite{ST,Tjde}, we get an a priori estimate with a loss of derivatives from the data. That is, in the future nonlinear analysis we will have to use Nash-Moser iterations for compensating this loss of derivatives.

In Nash-Moser iterations the basic state \eqref{21} playing the role of an intermediate state $(U_{n+1/2},\mathcal{H}_{n+1/2},\varphi_{n+1/2})$ (see \cite{ST,T09,Tcpam}) should finally converge to a solution of the nonlinear problem. Therefore, it is natural to assume that the basic state \eqref{21} besides \eqref{23}--\eqref{27} also satisfies the second boundary condition in \eqref{17} expressing the continuity of the total pressure. Moreover, recall that $\hat{\rho}|_{x_1=0}>0$ (see \eqref{23}). For the most usable case of a polytropic gas for which $\rho>0$ implies $p>0$, the continuity of the total pressure together with the condition $\hat{p}|_{x_1=0}>0$ give $\widehat{\mathcal{H}}|_{x_1=0}\neq 0$.
That is, it is natural to assume that the unperturbed vacuum magnetic field is non-zero on the boundary, i.e. $\widehat{\mathcal{H}}|_{x_1=0}\neq 0$. On the other hand, this assumption and the collinearity condition \eqref{collin} imply  $\widehat{H}|_{x_1=0}\neq 0$ that is, of course, a restriction on the unperturbed plasma magnetic field. It is still an open problem how to derive the a priori estimate \eqref{est} for the case when $\widehat{H}|_{x_1=0}=0$ (or it vanishes somewhere on the boundary) and $\widehat{\mathcal{H}}|_{x_1=0}\neq 0$ (we again refer to Section \ref{s5} for further discussions).

Regarding the case when the unperturbed plasma and vacuum magnetic fields are collinear on the boundary but the Rayleigh-Taylor sign condition \eqref{39} fails, we have the following Theorem for the linearized problem with frozen coefficients.

\begin{theorem}
The linearized plasma-vacuum interface problem with frozen coefficients (see problem \eqref{frozen}) is ill-posed if and only if the constant magnetic fields $\widehat{H}=(0,\widehat{H}_2,\widehat{H}_3)$ and $\widehat{\mathcal{H}}=(0,\widehat{\mathcal{H}}_2,\widehat{\mathcal{H}}_3)$ are collinear,
\begin{equation}
\widehat{H}_2\widehat{\mathcal{H}}_3-\widehat{H}_3\widehat{\mathcal{H}}_2=0,
\label{collin-fr}
\end{equation}
and the constant (frozen) coefficient $[\partial_1\hat{q}]$ is negative (i.e. when the Rayleigh-Taylor sign condition fails):
\begin{equation}
[\partial_1\hat{q}]<0.
\label{antiRT}
\end{equation}
\label{t2.2}
\end{theorem}

Note that, unlike Theorem \ref{t2.1}, in Theorem \ref{t2.2} one of the magnetic fields  $\widehat{H}$ and $\widehat{\mathcal{H}}$ or even both can be zero.\footnote{To say the truth, the proof of Theorem \ref{t2.1} stays valid for $\widehat{\mathcal{H}}|_{x_1=0}\equiv 0$ and any $\widehat{H}$.}

\section{Basic a priori estimate under the Rayleigh-Taylor sign \\condition}
\label{s3}

\subsection{Preparatory estimates}

We now derive some preparatory estimates which are important for the proof of the basic a priori estimate \eqref{est}. Following \cite{Tjde}, we first introduce the scalar potential $\xi$ for the div-curl system \eqref{36'bnew} (see also \cite{ST}; here we keep the notations from \cite{ST}):
\begin{equation}
\mathfrak{H}=\nabla\xi \,.
\label{potent}
\end{equation}
Then, it follows from \eqref{36'bnew}, \eqref{37'new} and \eqref{37'anew} that  $\xi$ satisfies the Neumann-Dirichlet problem (for freezed time $t$ and a given interface perturbation $\varphi$):
\begin{subequations}\label{ND}
\begin{alignat}{2}
\displaystyle
&  {\rm div}\,(A\nabla\xi)=0 &\mbox{in}&\ \Omega^-, \label{ND1}
\\
&  (A\nabla\xi)_1 =-\partial_2\bigl(\widehat{\mathcal{H}}_2\varphi \bigr) -\partial_3\bigl(\widehat{\mathcal{H}}_3\varphi \bigr) &\mbox{on}&\ \Gamma,
\label{ND2}
\\
&  \xi=0  &\qquad\mbox{on}&\ \Gamma_-,\label{ND3}
\\
&  (x_2,x_3)\rightarrow \xi (t,x_1,x_2,x_3) &\mbox{is}&\ \mbox{1-periodic} ,\label{ND4}
\end{alignat}
\end{subequations}
where $A=A(\nabla\widehat\Psi)=(\duno\widehat\Phi_1)^{-1}\hat \eta \, \hat \eta^T$ and $(A\nabla\xi)_1=-\mathcal{H}_N$ is the first component of the vector
$A\nabla\xi =\mathfrak{h}$, with $\hat\eta$ defined in \eqref{defeta}  .

\begin{lemma}
Solutions of problem \eqref{ND}\footnote{In fact, the boundary condition \eqref{ND2} can be ignored here as well as in Lemma \ref{l2}.} satisfy
\begin{equation}
\int_{\Gamma}\xi\mathcal{H}_N\,{\rm d}x'=\int_{\Omega^-}\duno\widehat\Phi_1|\mathcal{H}|^2\,{\rm d}x,
\label{l1.1}
\end{equation}
\begin{equation}
\int_{\Gamma}{\rm coeff}\,\xi\mathcal{H}_N\,{\rm d}x'\leq C\|\mathcal{H}\|^2_{L^2(\Omega^-)},
\label{l1.2}
\end{equation}
\begin{equation}
\int_{\Gamma}{\rm coeff}\,\xi H_N\,{\rm d}x'\leq C\left\{\|\mathcal{H}\|^2_{L^2(\Omega^-)}+\|{H}\|^2_{L^2(\Omega^+)}\right\},
\label{l1.3}
\end{equation}
with $\mathcal{H}$ defined through \eqref{potent} and $H$ satisfying \eqref{93}, and where {\rm coeff} is a coefficient depending smoothly on the basic state \eqref{21}. Moreover, here and later on $C$ is a positive constant that can change from line to line, and it may depend on other constants, in particular, in \eqref{l1.2} and \eqref{l1.3} the constant $C$ depends on $K$ (see \eqref{22}).
\label{l1}
\end{lemma}

\begin{proof}
Multiplying \eqref{ND1} by $\xi$, we easily get:
\[
\div(\xi A\nabla\xi ) = A\nabla\xi\cdot\nabla\xi = \hat\eta\mathcal{H} \cdot A^{-1}\hat\eta\mathcal{H}=\duno\widehat\Phi_1|\mathcal{H}|^2.
\]
Integrating the last identity over the domain $\Omega^-$ and taking into account \eqref{ND3} and \eqref{ND4}, we obtain \eqref{l1.1}.
If the coefficient ${\rm coeff}= {\rm coeff}(t,x)$ is defined in $\Omega^+$, we make the change from $x_1$ to $-x_1$ in it. Multiplying then \eqref{ND1} by ${\rm coeff}\,\xi$, integrating the result over the domain $\Omega^-$, using integration by parts, and taking again \eqref{ND3} and \eqref{ND4} into account, we come to the estimate
\begin{equation}
\int_{\Gamma}{\rm coeff}\,\xi\mathcal{H}_N\,{\rm d}x'\leq C\left\{\|\mathcal{H}\|^2_{L^2(\Omega^-)}+ \|\xi \|^2_{L^2(\Omega^-)}\right\},
\label{l1.2'}
\end{equation}
For $\xi$ having properties \eqref{ND3} and \eqref{ND4} the Poincar\'{e} inequality $\|\xi\|_{L^2(\Omega^-)}\leq C\|\nabla\xi\|_{L^2(\Omega^-)}$ applies. Using it, from \eqref{l1.2'} and \eqref{potent} we derive the desired estimate \eqref{l1.2}. At last, for obtaining estimate \eqref{l1.3} we multiply \eqref{93} by  ${\rm coeff}\,\xi$ (we make the change from $x_1$ to $-x_1$ in coeff if necessary) and then apply the same arguments as above.
\end{proof}

\begin{lemma}
Solutions of the div-curl system \eqref{36'bnew} and the boundary conditions \eqref{37'anew} satisfy the inequality
\begin{equation}
\int_{\Gamma}{\rm coeff}\,\mathcal{H}_N\mathcal{H}_{\tau_k}\,{\rm d}x'\leq C\|\mathcal{H}\|^2_{L^2(\Omega^-)}
\label{l2.1}
\end{equation}
for $k=2$ and $k=3$, where $C=C(K)>0$ is a constant.
\label{l2}
\end{lemma}

\begin{proof}
It is convenient (but not necessary) to use the scalar form \eqref{ND1} of the div-curl system \eqref{36'bnew}. Since the matrix $A$ is symmetric, we can derive the equality
\begin{equation}
A\nabla\xi\cdot\partial_k\xi= \frac{1}{2}\partial_k(A\nabla\xi\cdot\nabla\xi ) - \frac{1}{2}(\partial_kA)\nabla\xi\cdot\nabla\xi ,\qquad k=2,3.\label{l2.2}
\end{equation}
Multiplying then \eqref{ND1} by ${\rm coeff}\,\partial_k\xi$ and using \eqref{l2.2}, we obtain:
\begin{multline}
{\rm div} \left\{{\rm coeff}\,(\partial_k\xi) A\nabla\xi \right\} =(\partial_k\xi) A\nabla\xi\cdot\nabla{\rm coeff} + \frac{1}{2}\partial_k({\rm coeff}\,A\nabla\xi\cdot\nabla\xi )  \\[6pt] -\frac{1}{2}(\partial_k{\rm coeff})A\nabla\xi\cdot\nabla\xi - \frac{1}{2}{\rm coeff}(\partial_kA)\nabla\xi\cdot\nabla\xi .
\label{l2.3}
\end{multline}
Integrating \eqref{l2.3} over the domain $\Omega^-$, using the boundary conditions \eqref{37'anew} (see also \eqref{ND3} and \eqref{ND4}) and returning from $\xi$ to $\mathcal{H}$, we get \eqref{l2.1}.
\end{proof}

\subsection{Proof of Theorem \ref{t2.1}}

The a priori estimate \eqref{est} was managed to be derived in \cite{Tjde} only for the case of frozen coefficients. On the other hand, all of the main steps of the proof of estimate \eqref{est} except the final ones were performed in \cite{Tjde} for variable coefficients. Here, however, we  prefer to describe all the details of the proof and not just refer to \cite{Tjde} because, unlike \cite{Tjde}, the domains $\Omega^{\pm}$ are bounded and we have to modify several places in \cite{Tjde}. The more so, that the boundedness of the domain $\Omega^-$ plays a crucial role in ``closing'' the estimate \eqref{est} for variable coefficients because we have to use the Poincar\'{e} inequality for the potential $\xi$ (in fact, we have already used it in the proof of estimates \eqref{l1.2} and \eqref{l1.3} in Lemma \ref{l1}).

By standard arguments of the energy method we obtain for the hyperbolic system \eqref{34'anew}  the energy inequality
\begin{equation}
I(t)+ 2\int_{\omega_t}\mathcal{Q}\,{\rm d}x'{\rm d}s\leq C
\left( \| \mathcal{F}\|^2_{L^2(Q_T^+)} +\int_0^tI(s)\,{\rm d}s\right),
\label{55}
\end{equation}
where
\[
\omega_t =(-\infty ,t]\times\Gamma ,\quad
I(t)=\int_{\Omega^+}(\widehat{\mathcal{A}}_0\mathcal{U}\cdot\mathcal{U})\,{\rm d}x,\quad
\mathcal{Q}(s,x')=-\frac{1}{2}\bigl((\widehat{\mathcal{A}}_1+{\mathcal{E}}_{12})\mathcal{U}\cdot\mathcal{U}\bigr)|_{x_1=0},
\]
and, in view of \eqref{a10},
\[
\mathcal{Q}=-\frac{1}{2}({\mathcal{E}}_{12}\,\mathcal{U}\cdot\mathcal{U})|_{\omega_T}= -qv_N|_{\omega_T}
\]
While obtaining inequality \eqref{55} we used the boundary condition \eqref{37''anew} and the periodicity condition with respect to $x_2$ and $x_3$. Namely, the boundary integral
\[
\int_{\omega^+_t}\bigl((\widehat{\mathcal{A}}_1+{\mathcal{E}}_{12})\mathcal{U}\cdot\mathcal{U}\bigr)\,{\rm d}x'\,{\rm d}s
\]
(with $\omega^+_t =(-\infty ,t]\times\Gamma_+ $) vanishes thanks to \eqref{a10} and the boundary condition \eqref{37''anew} whereas the boundary integrals over the opposite sides of the unit $(x_2,x_3)$-square cancel out (below we will not comment the role of periodical boundary conditions anymore).

It follows from the boundary condition \eqref{35'dnew} that
\begin{equation}
\mathcal{Q}=-{q}{v}_N|_{\omega_T}=[\partial_1\hat{q}]\varphi {v}_N|_{\omega_T}
-(\widehat{\mathcal{H}}\cdot {\mathcal{H}}){v}_N|_{\omega_T}.
\label{56}
\end{equation}
In view of \eqref{35"},
\begin{align}
2[\partial_1\hat{q}]\varphi {v}_N|_{\omega_T} &= 2[\partial_1\hat{q}]\varphi (\partial_t\varphi+\hat{v}_2\partial_2\varphi+\hat{v}_3\partial_3\varphi -
\varphi\,\partial_1\hat{v}_{N})|_{\omega_T} \nonumber\\
 & = \partial_t\left\{ [\partial_1\hat{q}]\varphi^2\right\}+\partial_2\left\{ \hat{v}_2|_{\omega_T}\,[\partial_1\hat{q}]\varphi^2 \right\} +
\partial_3\left\{ \hat{v}_3|_{\omega_T}\,[\partial_1\hat{q}]\varphi^2  \right\}\nonumber
\\
 & \quad\; -\bigl\{\partial_t([\partial_1\hat{q}])+\partial_2(\hat{v}_2[\partial_1\hat{q}])+\partial_3(\hat{v}_3[\partial_1\hat{q}])
+2[\partial_1\hat{q}]\partial_1\hat{v}_{N}\bigr\}\bigr|_{\omega_T}\,\varphi^2.
\label{57}
\end{align}
Using the Rayleigh-Taylor sign condition \eqref{39}, from \eqref{55}--\eqref{57} we obtain
\begin{equation}
I(t)+ \epsilon \|\varphi (t)\|^2_{L^2(\Gamma )} + 2\int_{\omega_t}\widetilde{\mathcal{Q}}\,{\rm d}x'{\rm d}s \leq C
\left( \| \mathcal{F}\|^2_{L^2(Q_T^+)} +\int_0^t\left(I(s)
+\|\varphi (s)\|^2_{L^2(\Gamma )}\right){\rm d}s\right),
\label{58}
\end{equation}
where, in view of the boundary condition \eqref{35"},
\begin{equation}
\label{tQ}
\begin{split}
& \widetilde{\mathcal{Q}} =-(\widehat{\mathcal{H}}\cdot {\mathcal{H}}){v}_N|_{\omega_T}=\widetilde{\mathcal{Q}}_1+\widetilde{\mathcal{Q}}_2+\widetilde{\mathcal{Q}}_3,\\
& \widetilde{\mathcal{Q}}_1 =-(\widehat{\mathcal{H}}\cdot {\mathcal{H}})\partial_t\varphi|_{\omega_T},\quad
\widetilde{\mathcal{Q}}_2 =-(\hat{\mathcal{H}}\cdot {\mathcal{H}})(\hat{v}'\cdot\nabla'\varphi)|_{\omega_T},\quad
\widetilde{\mathcal{Q}}_3 =(\partial_1\hat{v}_N)(\hat{\mathcal{H}}\cdot {\mathcal{H}})\varphi|_{\omega_T},
\end{split}
\end{equation}
with $\hat{v}'=(\hat{v}_2,\hat{v}_3)$ and $\nabla'=(\partial_2,\partial_3)$.

Noticing that $\widehat{\mathcal{H}}\cdot\mathcal{H}=\hat{\mathfrak{h}}\cdot\mathfrak{H}$ and $\hat{\mathfrak{h}}_1|_{\omega_T}=\widehat{\mathcal{H}}_N|_{\omega_T}=0$  and passing to the potential $\xi$ (see \eqref{potent}), we rewrite $\widetilde{\mathcal{Q}}_1$:
\begin{equation}
\label{tQ1}
\begin{split}
& \widetilde{\mathcal{Q}}_1=-\partial_t\varphi \,(\widehat{\mathcal{H}}'\cdot\nabla'\xi )|_{\omega_T} = \widetilde{\mathcal{Q}}_{11}+\widetilde{\mathcal{Q}}_{12},\\
& \widetilde{\mathcal{Q}}_{11}= -\bigl(\partial_t(\varphi\widehat{\mathcal{H}}' )\cdot\nabla'\xi\bigr)|_{\omega_T},\quad \widetilde{\mathcal{Q}}_{12} = \varphi (\partial_t\widehat{\mathcal{H}}'\cdot\nabla'\xi )|_{\omega_T},
\end{split}
\end{equation}
where  $\widehat{\mathcal{H}}'=(\widehat{\mathcal{H}}_2,\widehat{\mathcal{H}}_3)$. Integrating by parts and using the boundary condition
\eqref{37'new} and equality \eqref{l1.1}, we get
\begin{multline}
\int_{\omega_t}\widetilde{\mathcal{Q}}_{11}\,{\rm d}x'{\rm d}s=
-\int_{\omega_t}\bigl(\partial_t(\varphi\widehat{\mathcal{H}}' )\cdot\nabla'\xi\bigr)\,{\rm d}x'{\rm d}s\\
=
\int_{\omega_t}\xi\partial_t\bigl({\rm div}'(\varphi\widehat{\mathcal{H}}')\bigr)\,{\rm d}x'{\rm d}s=\int_{\omega_t}\xi\partial_t\mathcal{H}_N\,{\rm d}x'{\rm d}s
=\int_{\Gamma}\xi\mathcal{H}_N\,{\rm d}x'-\int_{\omega_t}\mathcal{H}_N\partial_t\xi\,{\rm d}x'{\rm d}s \\= \int_{\Omega^-}\duno\widehat\Phi_1|\mathcal{H}|^2\,{\rm d}x-\int_{Q_t^-}\partial_1\bigl(\mathcal{H}_N\partial_t\xi\bigr)\,{\rm d}x{\rm d}s ,
\label{q11}
\end{multline}
where ${\rm div}'$ is the tangential divergence operator, i.e. ${\rm div}'\,b':=\partial_2b_2+\partial_3b_3$ for $b'=(b_2,b_3)$.

Taking into account \eqref{93} and \eqref{potent}, the integrand of the last integral in \eqref{q11} reads:
\begin{multline}
\partial_1\bigl(\mathcal{H}_N\partial_t\xi\bigr)=-\partial_t\xi\,{\rm div}'(\mathcal{H}'\duno\widehat\Phi_1) +\mathcal{H}_N\partial_t(\mathcal{H}_1\duno\widehat\Phi_1)\\ =-{\rm div}'\bigl( (\partial_t\xi )(\duno\widehat\Phi_1)\mathcal{H}'\bigr) +\duno\widehat\Phi_1 \bigl( \mathcal{H}'\cdot \partial_t(\mathcal{H}_1\nabla'\widehat{\Psi} +\mathcal{H}')\bigr) +\mathcal{H}_N\partial_t(\mathcal{H}_1\duno\widehat\Phi_1)\\
=-{\rm div}'\bigl( (\partial_t\xi )(\duno\widehat\Phi_1)\mathcal{H}'\bigr) +\frac{1}{2}\partial_t\bigl(\duno\widehat\Phi_1|\mathcal{H}|^2\bigr)+\frac{R}{2},
\label{q11'}
\end{multline}
where
\begin{equation}
\label{rest}
R=(\partial_t\duno\widehat\Phi_1)(\mathcal{H}_1^2-|\mathcal{H}'|^2)+2\left( \left\{(\duno\widehat\Phi_1)(\partial_t\nabla'\widehat{\Psi})-(\partial_t\duno\widehat\Phi_1)(\nabla'\widehat{\Psi})\right\}\cdot\mathcal{H}'
\right)\mathcal{H}_1\leq C|\mathcal{H}|^2.
\end{equation}
It follows from \eqref{q11}--\eqref{rest} that
\begin{multline}
\label{q11"}
2\int_{\omega_t}\widetilde{\mathcal{Q}}_{11}\,{\rm d}x'{\rm d}s=\int_{\Omega^-}\duno\widehat\Phi_1|\mathcal{H}|^2\,{\rm d}x-\int_{Q_t^-}R\,{\rm d}x{\rm d}s \\ \geq \frac{1}{2}\|\mathcal{H} (t)\|^2_{L^2(\Omega^- )} -C\int_0^t\|\mathcal{H} (s)\|^2_{L^2(\Omega^- )} {\rm d}s
\end{multline}
(recall that $\duno\widehat\Phi_1\geq 1/2$).

The integral of $\widetilde{\mathcal{Q}}_{12}$ will be below treated together with that of  $\widetilde{\mathcal{Q}}_{3}=(\partial_1\hat{v}_N)(\widehat{\mathcal{H}}'\cdot\nabla'\xi )\varphi|_{\omega_T}$ (see \eqref{tQ} and \eqref{tQ1}), and we now consider $\widetilde{\mathcal{Q}}_{2}$:
\begin{align}
\label{tQ2}
\widetilde{\mathcal{Q}}_{2} &=- (\widehat{\mathcal{H}}'\cdot\nabla'\xi )(\hat{v}'\cdot\nabla'\varphi)|_{\omega_T} \nonumber\\
 & =-(\hat{v}'\cdot\nabla'\xi)(\widehat{\mathcal{H}}'\cdot\nabla'\varphi )|_{\omega_T}  -
(\hat{v}_2\widehat{\mathcal{H}}_3-\hat{v}_3\widehat{\mathcal{H}}_2)(\partial_3\xi\partial_2\varphi -\partial_2\xi\partial_3\varphi )|_{\omega_T} \nonumber\\ & =\widetilde{\mathcal{Q}}_{21}+\widetilde{\mathcal{Q}}_{22},
\end{align}
where
\begin{equation}
\label{tQ2'}
\begin{split}
 & \widetilde{\mathcal{Q}}_{21}=-(\hat{v}'\cdot\nabla'\xi)\,{\rm div}'(\varphi\widehat{\mathcal{H}}')|_{\omega_T},\\
 & \widetilde{\mathcal{Q}}_{22}=\left.\left((\hat{v}'\cdot\nabla'\xi)\,\varphi\, {\rm div}'\widehat{\mathcal{H}}'-
(\hat{v}_2\widehat{\mathcal{H}}_3-\hat{v}_3\widehat{\mathcal{H}}_2)\bigl(\partial_2(\varphi\partial_3\xi ) -\partial_3(\varphi \partial_2\xi )\bigr)\right)\right|_{\omega_T}.
\end{split}
\end{equation}
Using \eqref{potent}, the boundary condition \eqref{37'new} and inequality \eqref{l2.1}, we estimate the integral of $\widetilde{\mathcal{Q}}_{21}$:
\begin{equation}
\label{tQ21}
\int_{\omega_t}\widetilde{\mathcal{Q}}_{21}\,{\rm d}x'{\rm d}s= -\int_{\omega_t}(\hat{v}_2\mathcal{H}_N\mathcal{H}_{\tau_2}+\hat{v}_3\mathcal{H}_N\mathcal{H}_{\tau_3})\,{\rm d}x'{\rm d}s\geq
-C\int_0^t\|\mathcal{H} (s)\|^2_{L^2(\Omega^- )} {\rm d}s.
\end{equation}
Integration by parts gives
\begin{equation}
\label{tQ22}
\int_{\omega_t}\widetilde{\mathcal{Q}}_{22}\,{\rm d}x'{\rm d}s =\int_{\omega_t}\widehat{\mathcal{Q}}_{22}\,{\rm d}x'{\rm d}s,
\end{equation}
with
\[
\widehat{\mathcal{Q}}_{22}=
\left.\left( \partial_2 (\hat{v}_2\widehat{\mathcal{H}}_3-\hat{v}_3\widehat{\mathcal{H}}_2)\partial_3\xi - \partial_3 (\hat{v}_2\widehat{\mathcal{H}}_3-\hat{v}_3\widehat{\mathcal{H}}_2)\partial_2\xi +(\hat{v}'\cdot\nabla'\xi)\,{\rm div}'\widehat{\mathcal{H}}'\right)\right|_{\omega_t}\varphi .
\]

It follows from \eqref{58}--\eqref{tQ1} and \eqref{q11"}--\eqref{tQ22} that
\begin{multline}
I(t)+ \epsilon \|\varphi (t)\|^2_{L^2(\Gamma )} +\frac{1}{2}\|\mathcal{H} (t)\|^2_{L^2(\Omega^- )}\leq  \mathcal{N}(t) \\  +C
\left( \| \mathcal{F}\|^2_{L^2(Q_T^+)} +\int_0^t\left(I(s)
+\|\varphi (s)\|^2_{L^2(\Gamma )}+\|\mathcal{H} (s)\|^2_{L^2(\Omega^- )}\right){\rm d}s\right),
\label{58"}
\end{multline}
where
\begin{equation}
\mathcal{N}(t)=-2\int_{\omega_t}(\widetilde{\mathcal{Q}}_{12}+\widehat{\mathcal{Q}}_{22}+\widetilde{\mathcal{Q}}_{3})\,{\rm d}x'{\rm d}s =
-\int_{\omega_t}(\hat{a}'\cdot \nabla'\xi)\,\varphi \,{\rm d}x'{\rm d}s
\label{nt}
\end{equation}
and after simple algebra we find that
\[
\hat{a}'=2\left.\left( \partial_t\widehat{\mathcal{H}}'+(\hat{v}'\cdot\nabla')\widehat{\mathcal{H}}'-(\widehat{\mathcal{H}}'\cdot\nabla')\hat{v}'+
\widehat{\mathcal{H}}'{\rm div}\hat{u}\right) \right|_{\omega_t}.
\]
Assumption \eqref{26} and the collinearity condition \eqref{collin} implies
\begin{equation}
\hat{a}'=\hat{\kappa}\widehat{H}',
\label{a'}
\end{equation}
with $\hat{\kappa} =2\,({\rm d}\hat{\beta} /{\rm d}t)$.

Integrating \eqref{nt} by parts and using \eqref{a'} and \eqref{95'}, we obtain
\begin{equation}
\mathcal{N}(t)=\int_{\omega_t}(\widehat{H}'\cdot \nabla'\hat{\kappa} ) \xi\varphi\,{\rm d}x'{\rm d}s+\int_{\omega_t}\hat{\kappa}\xi H_N \,{\rm d}x'{\rm d}s\,.
\label{nt'}
\end{equation}
For estimating the first integral in \eqref{nt'} we apply the Poincar\'{e} inequality and \eqref{potent} (as in the proof of Lemma \ref{l1}). Using then inequality \eqref{l1.3} for the second integral in \eqref{nt'}, we get the estimate
\begin{equation}
\mathcal{N}(t)\leq C\int_0^t\left(\|{H}(s)\|^2_{L^2(\Omega^+)}+\|\mathcal{H} (s)\|^2_{L^2(\Omega^- )}
+\|\varphi (s)\|^2_{L^2(\Gamma )}\right){\rm d}s.\label{nt-est}
\end{equation}
Combining \eqref{58"} and \eqref{nt-est} gives
\begin{multline}
I(t)+ \|\varphi (t)\|^2_{L^2(\Gamma )} +\|\mathcal{H} (t)\|^2_{L^2(\Omega^- )}  \\  \leq C
\left( \| \mathcal{F}\|^2_{L^2(Q_T^+)} +\int_0^t\left(I(s)
+\|\varphi (s)\|^2_{L^2(\Gamma )}+\|\mathcal{H} (s)\|^2_{L^2(\Omega^- )}\right){\rm d}s\right),
\label{58^}
\end{multline}
where the constant $C$ depends particularly on $\epsilon$. At last, applying the Gronwall lemma and taking into account the positive definiteness of the matrix $\widehat{\mathcal{A}}_0$, from \eqref{58^} we derive the desired a priori estimate \eqref{est}. That is, the proof of Theorem \ref{t2.1} is complete.

\begin{remark}
{\rm
Assume that the Rayleigh-Taylor sign condition \eqref{39} is satisfied together with the non-collinearity condition
\begin{equation}
|\widehat{H}\times\widehat{\mathcal{H}}|\geq \delta >0 \qquad \mbox{on}\  \omega_T
\label{non-collin}
\end{equation}
(see \eqref{non-coll}). Then there exist such {\it bounded} constants $\hat{\kappa}_1$ and $\hat{\kappa}_2$ that $\hat{a}'$ appearing in \eqref{nt} can be represented as
\begin{equation}
\hat{a}'=\hat{\kappa}_1\widehat{H}'+\hat{\kappa}_2\widehat{\mathcal{H}}'
\label{a'-non}
\end{equation}
on $\omega_T$ (cf. \eqref{a'}).
Using then the same arguments as in \eqref{nt'} and \eqref{nt-est} together with inequality \eqref{l1.2}, we again obtain  \eqref{nt-est} which gives us the a priori estimate \eqref{est}. That is, if the non-collinearity condition from \cite{ST,Tjde} is supplemented with the Rayleigh-Taylor sign condition, then we can derive even an $L^2$ basic a priori estimate for the linearized problem. Note that under the fulfilment of the non-collinearity condition but without condition \eqref{39} we are not able to ``close'' the a priori estimate in $L^2$ and have to prolong systems \eqref{34'anew} and \eqref{36'bnew} up to first-order derivatives of $\mathcal{U}$ and $\mathcal{H}$ (see \cite{ST,Tjde} for more details).
}
\label{r3.1}
\end{remark}

\section{Ill-posedness for simultaneous failure of non-collinearity and Rayleigh-Taylor sign conditions}
\label{s4}

\subsection{Rayleigh-Taylor instability detected as ill-posedness for frozen coefficients}

The a priori estimate \eqref{est} for the linearized plasma-vacuum interface problem for which the unperturbed flow satisfies the collinearity condition \eqref{collin} and the Rayleigh-Taylor sign condition \eqref{39} is just the first step towards the proof of the local-in-time existence and uniqueness theorem for the original nonlinear problem when the non-collinearity condition from \cite{ST,Tjde} is violated in some points/regions of the initial interface. We refer the reader to Section \ref{s5} for additional discussions and open problems in this direction. But now there appears the natural question: What happens if both the non-collinearity and Rayleigh-Taylor sign conditions fail in some points/regions of the initial interface? Our hypothesis is that the plasma-vacuum interface problem is not well-posed in this case. However, it is rather difficult and not really necessary to show this on the original nonlinear level and we restrict ourselves to the consideration of the linearized problem.\footnote{We know only one example \cite{Ebin} of the justification of the ill-posedness of a similar but much simpler nonlinear free boundary problem in fluid dynamics. This is the free boundary problem for the incompressible Euler equations with a vacuum boundary condition mentioned in Section \ref{s1}.}

Since the Rayleigh-Taylor sign condition \eqref{39} and the non-collinearity condition \eqref{non-collin} are both conditions on the unperturbed plasma-vacuum interface, using the localization argument, for showing the ill-posedness of problem \eqref{34'new} under the simultaneous failure of these conditions we can consider the linearized problem in the whole space $\mathbb{R}^3$. That is, we consider problem \eqref{34'new} without the (homogeneous) wall boundary conditions \eqref{37''anew} and \eqref{37'anew} and the periodicity conditions:
\begin{subequations}\label{a34'new}
\begin{alignat}{2}
\displaystyle
& \widehat{\mathcal{A}}_0\partial_t{\mathcal{U}}+\sum_{j=1}^{3}(\widehat{\mathcal{A}}_j+{\mathcal{E}}_{1j+1})\partial_j{\mathcal{U}}+
\widehat{\mathcal C}'{\mathcal{U}}=\mathcal{F}
  &\qquad\mbox{in}&\ (-\infty , T]\times \mathbb{R}^3_+,\label{a34'anew}
\\
&  \nabla\times {\mathfrak{H}}=0,\quad {\rm div}\,{\mathfrak{h}}=0 &\qquad\mbox{in}&\ (-\infty , T]\times \mathbb{R}^3_-, \label{a36'bnew}
\\
&  \partial_t\varphi=v_N-\hat{v}_2\partial_2\varphi-\hat{v}_3\partial_3\varphi +
\varphi\,\partial_1\hat{v}_{N},   & & \label{a35"}
\\
&  {q}=\widehat{\mathcal{H}}\cdot{\mathcal{H}}-  [ \partial_1\hat{q}] \varphi ,  & &\label{a35'dnew}
\\
&  {\mathcal{H}}_{N} =\partial_2\bigl(\widehat{\mathcal{H}}_2\varphi \bigr) +\partial_3\bigl(\widehat{\mathcal{H}}_3\varphi \bigr) &\qquad\mbox{on}&\ (-\infty , T]\times \{x_1=0\}\times\mathbb{R}^2,
\label{a37'new}
\\
&  (\mathcal{U},{\mathcal{H}},\varphi )=0 & \qquad\mbox{for}&\ t<0.\label{a38'fnew}
\end{alignat}
\end{subequations}
In fact, the same problem was studied in \cite{Tjde}. Apart from the main boundary conditions \eqref{a35"}--\eqref{a37'new} we again have the boundary constraint
\begin{equation}
{H}_{N}=\widehat{H}_2\partial_2\varphi +\widehat{H}_3\partial_3\varphi -
\varphi\,\partial_1\widehat{H}_{N} \quad\mbox{on}\ (-\infty , T]\times \{x_1=0\}\times\mathbb{R}^2
\label{95b}
\end{equation}
(cf. \eqref{95})
as well as the divergence constraint from \eqref{93} is now satisfied in the half-space $\mathbb{R}^3_+$:
\begin{equation}
 {\rm div}\,{h}=0 \quad\mbox{in}\  (-\infty , T]\times \mathbb{R}^3_+.
\label{93b}
\end{equation}

A direct proof that problem \eqref{a34'new} is ill-posed under the simultaneous failure of conditions \eqref{39} and \eqref{non-collin} is still difficult. On the other hand, the ill-posedness of the corresponding frozen coefficients problem indirectly points out that the variable coefficients problem cannot be well-posed. Indeed, for a wide class of hyperbolic initial boundary value problems for the well-posedness of a variable coefficients linear problem it is necessary and often sufficient that all frozen (constant) coefficients problems are well-posed (here we just refer to the classical paper of Kreiss \cite{Kreiss} and all other references to subsequent results can be found, for example, in \cite{BS,BThand,Maj84,Met}). We do not want to go into details of corresponding results and just notice that the fulfilment of the so-called Lopatinskii (or Kreiss-Lopatinskii \cite{Kreiss}) condition by a constant coefficients problem is necessary for its well-posedness, and the violation of the (weak) Lopatinskii condition is equivalent to the existence of Hadamard-type ill-posedness examples for this problem.

On the other hand, our problem \eqref{a34'new} is not purely hyperbolic as those studied, for example, in \cite{Kreiss} (see also references in \cite{BS,BThand,Maj84,Met}) and its boundary conditions \eqref{a35"}--\eqref{a37'new} contain the unknown front (to be exact, the front/interface perturbation) $\varphi$. In all fairness we should note that, for example,  linearized stability problems for shock waves \cite{BS,BThand,Maj84,Met} also contain an unknown front but, unlike our problem \eqref{a34'new}, the front symbol for shock waves is always elliptic, i.e. one can resolve the boundary conditions for the gradient $(\partial_t\varphi , \partial_2\varphi , \partial_3\varphi ) $. This is crucial for carrying the well-posedness results obtained for the case of constant coefficients under the fulfilment of the uniform Lopatinskii condition (the case of uniform or strong stability) or just the weak Lopatinskii condition (the case of neutral or weak stability) over variable coefficients.

Regarding free boundary problems with characteristic boundaries (e.g. problems for characteristic discontinuities), also for them the weak Lopatinskii condition satisfied for the constant coefficients problem is usually sufficient for the well-posedness of the corresponding variable coefficients problem if the front symbol is elliptic. At least this is so for vortex sheets \cite{CS2} and current-vortex sheets \cite{T09} as well as for our plasma-vacuum problem for the case when the non-collinearity condition \eqref{non-collin} holds \cite{ST}. Condition \eqref{non-collin} enables one to resolve \eqref{a37'new} and \eqref{95b} for  $\partial_2\varphi$ and $\partial_3\varphi$, and then we calculate $\partial_t\varphi$ by using \eqref{a35"}, i.e. the interface symbol is elliptic.

If the non-collinearity condition \eqref{non-collin}  fails, the interface symbol is {\it not elliptic}. This leads to a loss of ``control on the boundary". Then regardless of the fact that the constant coefficients linearized problem always satisfies the weak Lopatinskii condition, the corresponding variable coefficients problem is not unconditionally well-posed, and an {\it extra} condition is necessary for well-posedness. In our case such an extra condition is the Rayleigh-Taylor sign condition \eqref{39}.

Note that the Lopatinskii condition is usually being checked for the constant coefficients linearized problem resulting from the linearization about a  constant solution associated with the planar front $x_1=0$. Then, unlike \eqref{a35"}--\eqref{a37'new}, zero-order terms in $\varphi$ do not appear in the linearized boundary conditions (because the derivatives of constant solutions are zeros). If the uniform Lopatinskii condition holds, one can show that the front symbol is always elliptic (at least this was rigorously proved for shock waves \cite{Met}). Moreover, in this case the boundary conditions are robust against zero-order terms in $\varphi$. But, even if the Lopatinskii condition holds in a weak sense and the front symbol is elliptic, the zero-order terms in $\varphi$ again play a passive role in the analysis for variable coefficients. This was proved for weakly stable shock waves \cite{Col}, but also for characteristic discontinuities the known examples are vortex sheets  \cite{CS2} and current-vortex sheets \cite{T09}.  This is also so for our problem in the case when the non-collinearity condition \eqref{non-collin} holds, i.e. the front symbol is elliptic.\footnote{The fact that the weak Lopatinskii condition holds when the non-collinearity condition \eqref{non-collin} is satisfied follows from the a priori estimates proved in \cite{Tjde}.}

When, however, the front symbol is not elliptic, the zero-order terms in $\varphi$ may play a crucial role for the well-posedness of a frozen coefficients linearized free boundary problem. Below we first demonstrate this for such classical examples as the free boundary problems for the incompressible and compressible Euler equations with the vacuum boundary condition $p|_{\Gamma}=0$. Now we formulate a frozen coefficients version of problem \eqref{a34'new}, and then by analogy with it we just write down frozen coefficients problems for the mentioned problems for the Euler equations.

That is, we now freeze the coefficients in \eqref{a34'new}. Moreover, for technical simplicity and without loss of generality we below consider the case of a planar unperturbed interface by assuming that $\hat{\varphi}=0$. Freezing the coefficients at a point of the boundary $x_1=0$, we get the following constant coefficients problem:
\begin{subequations}\label{a34'new_2}
\begin{alignat}{2}
\displaystyle
& \widehat{\mathcal{A}}_0\partial_t{U}+\sum_{j=1}^{3}{\mathcal{E}}_{1j+1}\partial_j{U}+
\widehat{\mathcal C}'{U}=0
  &\qquad\mbox{in}&\ \mathbb{R}_+\times \mathbb{R}^3_+,\label{a34'anew_2}
\\
&  \nabla\times {\mathcal{H}}=0,\quad {\rm div}\,{\mathcal{H}}=0 &\qquad\mbox{in}&\ \mathbb{R}_+\times \mathbb{R}^3_-, \label{a36'bnew_2}
\\
&  \partial_t\varphi+\hat{v}_2\partial_2\varphi+\hat{v}_3\partial_3\varphi=v_1 +\hat{a}_0
\varphi,   & & \label{a35"_2}
\\
&  {q}=\widehat{\mathcal{H}}_2\mathcal{H}_2+\widehat{\mathcal{H}}_3\mathcal{H}_3+\hat{a} \varphi ,  & &\label{a35'dnew_2}
\\
&  {\mathcal{H}}_{1} =\widehat{\mathcal{H}}_2\partial_2\varphi  +\widehat{\mathcal{H}}_3\partial_3\varphi + \hat{a}_1\varphi &\qquad\mbox{on}&\ \mathbb{R}_+\times \{x_1=0\}\times\mathbb{R}^2,
\label{a37'new_2}
\\
&  (U,{\mathcal{H}},\varphi )= (U_0,{\mathcal{H}}_0,\varphi_0 ) & \qquad\mbox{for}&\ t=0,\label{a38'fnew_2}
\end{alignat}
\end{subequations}
where
\begin{equation}
\hat{a}=- [ \partial_1\hat{q}], \qquad \hat{a}_0=\partial_1\hat{v}_{1}\qquad\mbox{and}\quad \hat{a}_1=-\partial_1\widehat{\mathcal{H}}_1
\label{frcoeff}
\end{equation}
are constants as well as all the rest ``hat" values in \eqref{a34'new_2}, in particular, the matrices $\widehat{\mathcal{A}}_0$ and $\widehat{\mathcal C}'$ are constant coefficients matrices. Since below our goal will be the construction of particular exponential solutions, it is natural that here we drop the source term $\mathcal{F}$, introduce non-zero initial data and consider the problem for all times $t>0$.

In fact,  we can also omit the zero-order term $\widehat{\mathcal C}'{\mathcal{U}}$ in system \eqref{a34'anew_2} because its presence is not important for the process of construction of an Hadamard-type ill-posedness example. Indeed, this example is a sequence of solutions proportional to
\[
\exp \left\{ n\left(s t +\lambda^{\pm}x_1+ i(\omega ', x') \right)\right\}\qquad \mbox{for}\ \pm x_1>0,
\]
with  $n=1,2,3,\ldots$ and
\begin{equation}
\Re\,s>0,\quad \Re\,\lambda^+<0,\quad \Re\,\lambda^->0,\label{re_re}
\end{equation}
where $s$, $\lambda^+$ and $\lambda^-$ are complex constants, $\omega '=(\omega_2,\omega_3)$ and $\omega_{2,3}$ are real constants. Substituting these exponential solutions into system \eqref{a34'anew_2}, we get the dispersion relation
\[
\det \left(s \widehat{\mathcal{A}}_0 +\lambda^+{\mathcal{E}}_{12}+i\omega_2{\mathcal{E}}_{13}+i\omega_3 {\mathcal{E}}_{14}+ \frac{1}{n}\,\widehat{\mathcal C}'\right)=0
\]
from which the roots $\lambda^+= \lambda^+(s,\omega')$ can be found as a series in $1/n$:
\[
\lambda^+=\lambda^+_{(0)}+\frac{1}{n}\,\lambda^+_{(1)}+\frac{1}{n^2}\,\lambda^+_{(2)}+\ldots ,
\]
where $s=s_{0}+(1/n)s_{1}+\ldots $ and the leading term $\lambda^+_{(0)}$ does not depend on $\widehat{\mathcal C}'$. That is, without loss of generality we can drop the zero-order term $\widehat{\mathcal C}'{\mathcal{U}}$ in \eqref{a34'anew_2}.

System \eqref{a34'anew_2} with $\widehat{\mathcal C}'=0$ is the magnetoacoustics system, more precisely, it is the result of linearization of the MHD equations (in the half-space $x_1>0$) about a constant solution $\widehat{U}$ with $\hat{v}_1=0$ and $\widehat{H}_1=0$. Since $\hat{v}_2$ and $\hat{v}_3$ are now constants, we can apply a Galilean transformation so that the operator
$\partial_t+\hat{v}_2\partial_2 +\hat{v}_3\partial_3$
appearing in the magnetoacoustics system as well as in the boundary condition \eqref{a35"_2} becomes $\partial_t$. That is, without loss of generality we can suppose that $\hat{v}_2=\hat{v}_3=0$. The final form of the frozen coefficients problem in terms of the scalar potential \eqref{potent} reads:
\begin{subequations}\label{frozen}
\begin{alignat}{2}
\displaystyle
& \partial_t \bigl(q-\widehat{H}_2H_2-\widehat{H}_3H_3\bigr)+\hat{\rho}\hat{c}^2{\rm div}\,v =0, & &\label{frozen_MHD1}
\\
& \hat{\rho}\,\partial_tv -\ell^+H+\nabla q=0, & & \label{frozen_MHD2}
\\
& \partial_tH -\ell^+v+ \widehat{H}{\rm div}\,v =0   &\qquad\mbox{in}&\ \mathbb{R}_+\times \mathbb{R}^3_+,\label{frozen_MHD3}
\\
&  \triangle\xi=0 &\qquad\mbox{in}&\ \mathbb{R}_+\times \mathbb{R}^3_-, \label{frozen_Laplace}
\\
&  \partial_t\varphi=v_1 +\hat{a}_0
\varphi, \quad
 {q}=\ell^-\xi +\hat{a} \varphi ,  & &\label{frozen_bound2}
\\
&  \partial_1\xi =\ell^-\varphi + \hat{a}_1\varphi &\qquad\mbox{on}&\ \mathbb{R}_+\times \{x_1=0\}\times\mathbb{R}^2,
\label{frozen_bound3}
\end{alignat}
\end{subequations}
with some initial data, where $\hat{c}=1/\sqrt{\hat{\rho}_p}$ is the constant sound speed, $\widehat{H}=(0,\widehat{H}_2,\widehat{H}_3)$,
\[
\ell^+=\widehat{{H}}_2\partial_2+\widehat{{H}}_3\partial_3\quad\mbox{and}\quad
\ell^-=\widehat{\mathcal{H}}_2\partial_2+\widehat{\mathcal{H}}_3\partial_3.
\]
Here we do not include the autonomous equation $\partial_tS=0$ for the entropy perturbation $S$ in system \eqref{frozen_MHD1}--\eqref{frozen_MHD3} because without loss of generality one can assume that $S\equiv 0$.

\paragraph{Incompressible Euler equations with a vacuum boundary condition} Now, without going into details, by analogy with problem \eqref{frozen} we write down the frozen coefficients linearized problem
for the free boundary problem for the incompressible Euler equations with the vacuum boundary condition $p|_{\Gamma}=0$:
\begin{subequations}\label{froz2D}
\begin{alignat}{2}
\displaystyle
& \partial_tv +\nabla p=0,  \quad
{\rm div}\,v=0 &\qquad\mbox{in}&\ \mathbb{R}_+\times \mathbb{R}^3_+,\label{froz2Db}
\\
&  \partial_t\varphi=v_1 +\hat{a}_0
\varphi, \quad  p= \hat{a} \varphi  &\qquad\mbox{on}&\ \mathbb{R}_+\times \{x_1=0\}\times\mathbb{R}^2,
\label{froz2Dd}
\end{alignat}
\end{subequations}
where $p$ is the scaled pressure (divided by the constant density $\hat{\rho}$), the constant $\hat{a}_0$ is the same as in \eqref{frcoeff}, and $\hat{a}=- \partial_1\hat{p}$.

Substituting the sequence
\begin{equation}
\begin{pmatrix} p_n \\ v_n \end{pmatrix} =\begin{pmatrix} \bar{p}_n \\ \bar{v}_n \end{pmatrix}\exp \left\{ n\left(s t +\lambda x_1+ i(\omega ', x') \right)\right\}
\label{exp_sol}
\end{equation}
into equations \eqref{froz2Db}, we get the dispersion relation
\[
s^2(\lambda^2-\omega^2)=0
\]
giving the root $\lambda = -\omega$ with $\Re\lambda <0$ (cf. \eqref{re_re}), where $(\bar{p}_n , \bar{v}_n)$ is a constant vector and $\omega =|\omega'|$. Without loss of generality we may assume that $\omega =1$, i.e. $\lambda =-1$. Then from the first equation in \eqref{froz2Db} we have
\begin{equation}
\bar{v}_{1n}=\frac{\bar{p}_{n}}{s},
\label{vbar}
\end{equation}
where $\bar{v}_{1n}$ is the first component of the vector $\bar{v}_n$.

In view of \eqref{vbar}, substituting \eqref{exp_sol} and
\begin{equation}
\varphi_n=\bar{\varphi}_n\exp \left\{ n\left(s t + i(\omega ', x') \right)\right\}
\label{exp_sol'}
\end{equation}
into the boundary conditions \eqref{froz2Dd}, we get for the constants $\bar{p}_{n}$ and $\bar{\varphi}_n$ the algebraic system
\[
\begin{pmatrix}
ns-\hat{a}_0 & \displaystyle\frac{1}{s}\\[6pt]
\hat{a} & 1
\end{pmatrix}
\begin{pmatrix}
\bar{\varphi}_n\\
\bar{p}_{n}
\end{pmatrix}=0
\]
which has a non-zero solution if
\begin{equation}
ns^2-\hat{a}_0s-\hat{a}=0.
\label{eq_s}
\end{equation}
For big enough $n$ and $\hat{a}\neq 0$, the last equation has the ``unstable" root
\begin{equation}
s=\frac{\hat{a}_0}{2n}+\sqrt{\frac{\hat{a}}{n}+\frac{\hat{a}_0^2}{4n^2}}=\frac{\sqrt{\hat{a}}}{\sqrt{n}}+\mathcal{O}\left(\frac{1}{n} \right)
\label{root_s}
\end{equation}
if and only if $\hat{a}>0$, i.e. if and only if the Rayleigh-Taylor sign condition
\begin{equation}
\partial_1\hat{p} >0
\label{RT_Eul}
\end{equation}
fails.

It is {\it very important} that root \eqref{root_s} for $\hat{a}>0$ after the substitution into the exponential sequences \eqref{exp_sol} and \eqref{exp_sol'} gives an infinite growth in time of order $\sqrt{n}$ as $n\rightarrow \infty$ for any fixed (even very small) $t>0$. This is not only usual exponential instability but {\it ill-posedness} (violent instability). Note that for $\hat{a}=0$  equation \eqref{eq_s} has the ``unstable" root $s=\hat{a}_0/n$ for $\hat{a}_0>0$ but it just gives exponential instability but not ill-posedness.

That is, the above simple normal modes analysis show that Rayleigh-Taylor instability  can be trivially detected as ill-posedness on the level of frozen coefficients. Note that a rigorous proof of the ill-posedness of the original nonlinear free boundary problem for the incompressible Euler equations with a vacuum boundary condition $p|_{\Gamma}=0$ under the violation of the Rayleigh-Taylor sign condition $(\partial p/\partial N)|_{\Gamma}\leq -\epsilon<0$ is, of course, a difficult mathematical problem (see \cite{Ebin}).

\paragraph{Compressible Euler equations with a vacuum boundary condition} For the case of compressible liquid ($\rho|_{\Gamma}>0$) the above simple calculations are just a little bit more involved. The ``compressible" counterpart of problem \eqref{froz2D} reads (cf. \eqref{frozen}):
\begin{subequations}\label{frozen_Eul}
\begin{alignat}{2}
\displaystyle
& \partial_t p+\hat{\rho}\hat{c}^2{\rm div}\,v =0,\quad
\hat{\rho}\,\partial_tv +\nabla p=0, &\qquad\mbox{in}&\ \mathbb{R}_+\times \mathbb{R}^3_+,\label{frozen_1}
\\
&  \partial_t\varphi=v_1 +\hat{a}_0
\varphi, \quad  p= \hat{a} \varphi &\qquad\mbox{on}&\ \mathbb{R}_+\times \{x_1=0\}\times\mathbb{R}^2,\label{frozen_bound}
\end{alignat}
\end{subequations}
where the constants $\hat{a}_0$ and $\hat{a}$ are the same as in \eqref{froz2D}, and $\hat{c}$ is the constant sound speed. For system \eqref{frozen_1} we easily find the dispersion relation
\[
s^2(\hat{c}^2\lambda^2-\hat{c}^2\omega^2-s^2)=0
\]
giving the root
\begin{equation}
\lambda = -\sqrt{1+\left(\frac{s}{\hat{c}}\right)^2},
\label{lam}
\end{equation}
where without loss of generality we again assume that $\omega =1$. Then from the second equation in \eqref{frozen_1} we have
\begin{equation}
\bar{v}_{1n}=\frac{\bar{p}_{n}}{\hat{\rho}s}\sqrt{1+\left(\frac{s}{\hat{c}}\right)^2}.
\label{vbar'}
\end{equation}

Substituting \eqref{exp_sol} and \eqref{exp_sol'} into the boundary conditions \eqref{frozen_bound} and using \eqref{vbar'}, we obtain for the constants $\bar{p}_{n}$ and $\bar{\varphi}_n$ an algebraic system which has a non-zero solution if (cf. \eqref{eq_s})
\begin{equation}
ns^2-\hat{a}_0s-\frac{\hat{a}}{\hat{\rho}}\sqrt{1+\left(\frac{s}{\hat{c}}\right)^2}=0.
\label{eq_s'}
\end{equation}
As for the case of incompressible fluid, we see that for $\hat{a}=0$ the final equation for the frequency $s$ has the ``unstable" root $s=\hat{a}_0/n$ for $\hat{a}_0>0$, but this just means exponential instability but not ill-posedness. Considering now the case $\hat{a}=-\partial_1\hat{p}\neq 0$ and expanding \eqref{eq_s'} in powers of $1/\sqrt{n}$, we find the root (cf. \eqref{root_s})
\[
s=\sqrt{\frac{\hat{a}}{\hat{\rho}}}\,\frac{1}{\sqrt{n}}+\mathcal{O}\left(\frac{1}{n} \right)
\]
which gives ill-posedness if and only if the Rayleigh-Taylor sign condition \eqref{RT_Eul} fails.

\subsection{Hadamard-type ill-posedness example for the frozen coefficients plasma-vacuum interface problem}

Before proceeding to our frozen coefficients problem \eqref{frozen} we consider the simpler case of incompressible fluid for which the linearized plasma-vacuum interface problem was studied in \cite{MoTraTre-vacuum}. The ``incompressible" counterpart of problem \eqref{frozen} reads:
\begin{subequations}\label{frozen_inc}
\begin{alignat}{2}
\displaystyle
& \partial_tv -\ell^+H+\nabla q=0,\quad {\rm div}\,v =0, & & \label{frozen_inc1}
\\
& \partial_tH -\ell^+v =0   &\qquad\mbox{in}&\ \mathbb{R}_+\times \mathbb{R}^3_+,\label{frozen_inc2}
\\
&  \triangle\xi=0 &\qquad\mbox{in}&\ \mathbb{R}_+\times \mathbb{R}^3_-, \label{frozen_inc3}
\\
&  \partial_t\varphi=v_1 +\hat{a}_0
\varphi,\quad   {q}=\ell^-\xi +\hat{a} \varphi ,  & &\label{frozen_bound2-inc}
\\
&  \partial_1\xi =\ell^-\varphi + \hat{a}_1\varphi &\qquad\mbox{on}&\ \mathbb{R}_+\times \{x_1=0\}\times\mathbb{R}^2,
\label{frozen_bound3-inc}
\end{alignat}
\end{subequations}
where the equations are written in a suitable scaled form for which $\hat{\rho}=1$.

We seek sequences of exponential solutions to problem \eqref{frozen_inc} in the form
\begin{equation}
\begin{pmatrix} q_n \\ v_n \\ H_n\end{pmatrix} =\begin{pmatrix} \bar{q}_n \\ \bar{v}_n \\ \bar{H}_n \end{pmatrix}\exp \left\{ n\left(s t +\lambda^+ x_1+ i(\omega ', x') \right)\right\},\qquad x_1>0,
\label{exp_solMHD}
\end{equation}
\begin{equation}
\xi_n =\bar{\xi}_n \exp \left\{ n\left(s t +\lambda^- x_1+ i(\omega ', x') \right)\right\},\qquad x_1<0,
\label{exp_solMHD'}
\end{equation}
\begin{equation}
\varphi_n=\bar{\varphi}_n\exp \left\{ n\left(s t + i(\omega ', x') \right)\right\},
\label{exp_sol"}
\end{equation}
where, as above, all the ``bar" values are complex constants and $s$ and $\lambda^{\pm}$ should satisfy \eqref{re_re}. As for the incompressible Euler equations, we easily find $\lambda^+=-1$ either from the dispersion relation for system \eqref{frozen_inc1}--\eqref{frozen_inc2} or just from the Laplace equation $\triangle q=0$ following from this system (we again assume without loss of generality that $\omega=1$). Quite analogously from the Laplace equation \eqref{frozen_inc3} we find $\lambda^-=1$.

The result of substitution of \eqref{exp_solMHD} with $\lambda^+=1$ into the first equations in \eqref{frozen_inc1} and \eqref{frozen_inc2},
\[
s\bar{v}_{1n}-i\hat{w}^+\bar{H}_{1n}-\bar{q}_{1n}=0,\quad s\bar{H}_{1n}-i\hat{w}^+\bar{v}_{1n}=0,
\]
implies
\begin{equation}
\bar{v}_{1n}=\frac{s}{s^2+(\hat{w}^+)^2}\,\bar{q}_{1n},
\label{qn}
\end{equation}
where $\hat{w}^+=\widehat{H}_2\omega_2+\widehat{H}_3\omega_3$. Substituting \eqref{exp_solMHD}--\eqref{exp_sol"} with $\lambda^{\pm}=\mp1$ into the boundary conditions \eqref{frozen_bound2-inc}, \eqref{frozen_bound3-inc} and using \eqref{qn}, we obtain the algebraic system
\[
\begin{pmatrix}
ns-\hat{a}_0 & \displaystyle \frac{s}{s^2+(\hat{w}^+)^2} &0 \\[6pt]
\hat{a} &1 & in\hat{w}^- \\
\hat{a}_1+in\hat{w}^- & 0 & -n
\end{pmatrix}
\begin{pmatrix}
\bar{\varphi}_n\\
\bar{q}_{n}\\
\bar{\xi}_{n}
\end{pmatrix}
=0,
\]
with $\hat{w}^-=\widehat{\mathcal{H}}_2\omega_2+\widehat{\mathcal{H}}_3\omega_3$, which has a non-zero solution if
\begin{equation}
s\left\{n\left(s^2+(\hat{w}^+)^2+(\hat{w}^-)^2\right) -(\hat{a}+i\hat{w}^-\hat{a}_1)\right\}-\hat{a}_0\left(s^2+(\hat{w}^+)^2\right)=0.
\label{eq-for-s}
\end{equation}

We will go ahead and say that, exactly as for problems \eqref{froz2D} and \eqref{frozen_Eul}, the coefficient $\hat{a}_0$ plays no role for the existence/nonexistence of roots $s$ giving ill-posedness. Therefore, we first assume that $\hat{a}_0=0$. Then, ignoring the neutral mode $s=0$, from \eqref{eq-for-s} we obtain
\begin{equation}
s^2=-\left((\hat{w}^+)^2+(\hat{w}^-)^2\right)+\frac{\hat{a}+i\hat{w}^-\hat{a}_1}{n}.
\label{root-s2}
\end{equation}
Recall that $\omega =1$, i.e. $\omega'\neq  0$ (for the 1D case $\omega' =0  $ we get $\lambda^{\pm} =0$). Then
\[
\widehat{W}=(\hat{w}^+)^2+(\hat{w}^-)^2=0
\]
if and only if $\hat{w}^+=\hat{w}^-=0$, i.e. when the vectors $\widehat{H}'$ and $\widehat{\mathcal{H}}'$ are both perpendicular to $\omega'$ or they are both zero or one of them is zero and another one is perpendicular to $\omega'$. In other words, $\widehat{W}=0$ for some $\omega'$ if and only if the vectors $\widehat{H}'$ and $\widehat{\mathcal{H}}'$ are collinear.

If $\widehat{W}\neq 0$, then the leading term in the right-hand side in \eqref{root-s2} is $-\widehat{W}$. In this case
\begin{equation}
s=\pm i\widehat{W}^{1/2}+\mathcal{O}\left(\frac{1}{n}\right) .
\label{root-s}
\end{equation}
One of these roots $s$ can be ``unstable", i.e. $\Re s>0$, but anyway $\Re s$ is of order $1/n$, i.e. it does not give an infinite growth of the exponential solutions \eqref{exp_solMHD}--\eqref{exp_sol"} for a fixed time $t> 0$ as $n\rightarrow \infty$. If the vectors $\widehat{H}'$ and $\widehat{\mathcal{H}}'$ are collinear, we choose $\omega'$ such that  $\hat{w}^+=\hat{w}^-=0$. Then, from \eqref{root-s2} we find the root
\[
s=\frac{\sqrt{\hat{a}}}{\sqrt{n}}
\]
giving an ill-posedness example if and only if
\[
\hat{a}=- [ \partial_1\hat{q}] >0,
\]
i.e. when the Rayleigh-Taylor sign condition $[ \partial_1\hat{q}] >0$ fails.

Let us now $\hat{a}_0\neq 0$. Expanding \eqref{eq-for-s} in powers of $1/\sqrt{n}$, for $\widehat{W}\neq 0$ we again find the roots $s$ in form \eqref{root-s} which does not imply ill-posedness. For collinear vectors $\widehat{H}'$ and $\widehat{\mathcal{H}}'$ we choose $\omega'$ such that $\hat{w}^+=\hat{w}^-=0$ ($\widehat{W}=0$). Then, ignoring the neutral mode $s=0$, from \eqref{eq-for-s} we obtain equation \eqref{eq_s} (with $\hat{a}=- [ \partial_1\hat{q}]$) which has root \eqref{root_s} giving an ill-posedness example if and only if $[ \partial_1\hat{q}] <0$.

\paragraph{Proof of Theorem \ref{t2.2}}
The process of construction of an Hadamard-type ill-posedness example for problem \eqref{frozen} is very close to that for its ``incompressible" counterpart \eqref{frozen_inc} but calculations are just a little bit more involved. Substituting \eqref{exp_solMHD} into the magnetoacoustics system \eqref{frozen_MHD1}--\eqref{frozen_MHD3} and omitting technical calculations, we get a dispersion relation giving the only root\footnote{In fact, we can just use calculations for compressible current sheets which are current-vortex sheets with no jump in the velocity (see \cite{T05} and references therein).}
\begin{equation}
\lambda^+=-\sqrt{1+\frac{s^4}{(\hat{c}^2+\hat{c}_{\rm A}^2)s^2+\hat{c}^2(\widetilde{w}^+)^2}}
\label{lam+}
\end{equation}
matching properties \eqref{re_re}, where $\hat{c}_{\rm A}=|\widehat{H}|/\sqrt{\hat{\rho}}$ is the Alfv\'{e}n velocity, $\widetilde{w}^+=\hat{w}^+/\sqrt{\hat{\rho}}$, and without loss of generality we assume that $\omega =1$. Clearly, the root $\lambda^+$ coincides with \eqref{lam} for $\widehat{H}=0$. From the Laplace equation \eqref{frozen_Laplace} we easily find $\lambda^-=1$.

Substituting \eqref{exp_solMHD} with $\lambda^+$ given by \eqref{lam+} into the first equations in \eqref{frozen_MHD2} and \eqref{frozen_MHD3}, we obtain
\begin{equation}
\bar{v}_{1n}=\frac{\bar{q}_{1n}s}{\hat{\rho}s^2+(\hat{w}^+)^2}\,\sqrt{1+\frac{s^4}{(\hat{c}^2+\hat{c}_{\rm A}^2)s^2+\hat{c}^2(\widetilde{w}^+)^2}}.
\label{qn-comp}
\end{equation}
Taking then \eqref{lam+}, $\lambda^-=1$ and \eqref{qn-comp} into account and substituting \eqref{exp_solMHD}--\eqref{exp_sol"} into the boundary conditions \eqref{frozen_bound2}, \eqref{frozen_bound2}, we arrive at the algebraic system
\[
\begin{pmatrix}
ns-\hat{a}_0 & \quad \displaystyle \frac{s}{\hat{\rho}s^2+(\hat{w}^+)^2}\,\sqrt{1+\frac{s^4}{(\hat{c}^2+\hat{c}_{\rm A}^2)s^2+\hat{c}^2(\widetilde{w}^+)^2}}\quad &0 \\[9pt]
\hat{a} &1 & in\hat{w}^- \\
\hat{a}_1+in\hat{w}^- & 0 & -n
\end{pmatrix}
\begin{pmatrix}
\bar{\varphi}_n\\
\bar{q}_{n}\\
\bar{\xi}_{n}
\end{pmatrix}
=0
\]
which has a non-zero solution if (cf. \eqref{eq-for-s})
\begin{multline}
s\Biggl\{n\left(\hat{\rho}s^2+(\hat{w}^+)^2+(\hat{w}^-)^2\sqrt{1+\frac{s^4}{(\hat{c}^2+\hat{c}_{\rm A}^2)s^2+\hat{c}^2(\widetilde{w}^+)^2}}\,\right) \\-(\hat{a}+i\hat{w}^-\hat{a}_1)\sqrt{1+\frac{s^4}{(\hat{c}^2+\hat{c}_{\rm A}^2)s^2+\hat{c}^2(\widetilde{w}^+)^2}}\,\Biggr\}-\hat{a}_0\left(\hat{\rho}s^2+(\hat{w}^+)^2\right)=0.
\label{eq-for-s'}
\end{multline}

Expanding equation \eqref{eq-for-s'} in powers of $1/\sqrt{n}$, we seek its roots as series
\[
s=s_0+\frac{1}{\sqrt{n}}\,s_1+\frac{1}{n}\,s_2+\frac{1}{n\sqrt{n}}\,s_3+\ldots \;.
\]
One can show that if $s_0\neq 0$, then $s_1=0$. Hence, a root $s$ with $\Re s_0 =0$ which could be, in principle, an ``unstable" root with $\Re s>0$ cannot give ill-posedness because $\Re s =\mathcal{O} (1/n)$. That is,  we should check whether there are roots $s$ with $\Re s_0 >0$. A non-zero $s_0$ must satisfy the equation
\begin{equation}
\hat{\rho}s_0^2+(\hat{w}^+)^2+(\hat{w}^-)^2\sqrt{1+\frac{s_0^4}{(\hat{c}^2+\hat{c}_{\rm A}^2)s_0^2+\hat{c}^2(\widetilde{w}^+)^2}}=0.
\label{s_0}
\end{equation}
Instead of a direct analysis of this equation we can just refer to the study in \cite{Tjde} of the frozen coefficients linearized problem by the energy method. One can check that the arguments in \cite{Tjde} giving an a priori $L^2$ estimate are still valid for the case when $\hat{a}=\hat{a}_0=\hat{a}_1=0$. For this case resulting from the linearization of the nonlinear problem about a constant solution  $(\widehat{U},\widehat{\mathcal{H}},0)$ spectral analysis gives the equation for $s$ which coincides with \eqref{s_0}. But, since the presence of an a priori $L^2$ estimate proves the fulfilment of the Lopatinskii condition, i.e. the non-existence of unstable modes with $\Re s >0$, equation \eqref{s_0} cannot have roots with $\Re s_0 >0$.

That is, ill-posedness can happen if only $s_0=0$. It follows from \eqref{s_0} that $s_0=0$ if and only if $\hat{w}^+=\hat{w}^-=0$. Hence, as for incompressible MHD, ill-posedness can happen if only the vectors $\widehat{H}'$ and $\widehat{\mathcal{H}}'$ are collinear. Choosing then $\omega'$ such that $\hat{w}^+=\hat{w}^-=0$, from \eqref{eq-for-s'} we get the equation (cf. \eqref{eq_s'})
\begin{equation}
ns^2-\hat{a}_0s-\frac{\hat{a}}{\hat{\rho}}\,\sqrt{1+\frac{s^2}{\hat{c}^2+\hat{c}_{\rm A}^2}}=0.
\label{eq-for-s"}
\end{equation}
As for equation \eqref{eq_s'}, for $\hat{a}=0$ we have the ``unstable" root $s=\hat{a}_0/n$ if $\hat{a}_0>0$, but the existence of such a root does not imply ill-posedness. Assuming that $\hat{a}\neq 0$ and expanding \eqref{eq-for-s"} in powers of $1/\sqrt{n}$, we find the root \[
s=\sqrt{\frac{\hat{a}}{\hat{\rho}}}\,\frac{1}{\sqrt{n}}+\mathcal{O}\left(\frac{1}{n} \right)
\]
which gives ill-posedness if and only if $\hat{a}<0$, i.e. \eqref{antiRT} holds. The proof of Theorem \ref{t2.2} is thus complete.

\begin{remark}
{\rm
Probably from the physical point of view the ``mathematical" requirement of the ellipticity of the interface symbol appearing for our problem as the non-collinearity condition is not quite clear. The above simple calculations of the normal modes analysis clarify the physical sense of the non-collinearity condition. Indeed, if the tangential magnetic fields $\widehat{H}'$ and $\mathcal{\widehat{H}}'$ are not collinear, then they cannot be both collinear with the tangential wave vector $\omega'$ (to be exact, the tangential component of the wave vector is $n\omega'$, see \eqref{exp_solMHD}--\eqref{exp_sol"}).\footnote{Since the first components of the constant vectors $\widehat{H}$ and $\mathcal{\widehat{H}}$ are zeros, we can say the same about the magnetic fields and the (3D) wave vector in a general direction.}
}
\label{r4.1}
\end{remark}

\section{Open problems}
\label{s5}

The results obtained for the plasma-vacuum interface problem \eqref{4}, \eqref{6}, \eqref{9}--\eqref{8'} in \cite{ST} and in this paper (Theorems \ref{t2.1} and \ref{t2.2}) encourage us to prove the local-in-time existence and uniqueness theorem in suitable Sobolev spaces and under the hyperbolicity conditions \eqref{5} and suitable compatibility conditions satisfied for the initial data \eqref{8'}, provided that at each point of the initial interface $\Gamma (0)$ the plasma density is strictly positive and the Rayleigh-Taylor sign condition \eqref{RT} is satisfied at all those points of $\Gamma (0)$  where the non-collinearity condition \eqref{non-coll} fails. Before the proof of such a most general theorem one can try to get a little bit weaker result by assuming that the Rayleigh-Taylor sign condition is satisfied at each point of the initial interface without any restrictions on the relative position of the plasma and vacuum magnetic fields on $\Gamma (0)$. The first step towards the proof of the corresponding (``weakened") theorem for the nonlinear problem \eqref{4}, \eqref{6}, \eqref{9}--\eqref{8'} should be an a priori estimate for the linearized problem \eqref{34'new} under the Rayleigh-Taylor sign condition \eqref{39} satisfied for the basic state. That is, we need to remove the collinearity condition \eqref{collin} from the list of our assumptions in Theorem \ref{t2.1}. The more so that the form of such a condition itself is unstable with respect to time evolution, i.e. it cannot be guaranteed to be satisfied for the nonlinear problem for short times if it was true at the first moment.

At the same time, if the Rayleigh-Taylor sign condition \eqref{39} is supplemented with the non-collinearity condition \eqref{non-collin} for the basic state, then we can derive for the linearized problem the same $L^2$ a priori estimate as in Theorem \ref{t2.1} by using the same arguments as we used for the case when \eqref{39} is supplemented with the collinearity condition \eqref{collin} (see Remark \ref{r3.1}). However, it is still an open problem how to get an a priori estimate if the unperturbed plasma and vacuum magnetic fields are collinear in some regions of the boundary $\omega_T$ and/or even in some isolated points of $\omega_T$. At last, we note that the proof of the existence of solutions to the linearized problem for which the basic state satisfies the Rayleigh-Taylor sign condition \eqref{39} is also an open problem for the future work.

\vspace*{7mm}

{\bf E-mail:} trakhinin@mail.ru


\begin{thebibliography}{100}

\bibitem{Al} S. Alinhac, Existence d'ondes de rar\'{e}faction pour des syst\`{e}mes
quasi-lin\'{e}aires hyperboliques multidimensionnels. {\it Comm. Partial Differential Equations} \textbf{14} (1989), 173--230

\bibitem{BS}
S. Benzoni-Gavage, D. Serre, {\it Multidimensional hyperbolic partial differential equations. First-order systems and applications}. Oxford University Press, Oxford, 2007

\bibitem{BThand}
A. Blokhin, Y. Trakhinin, Stability of strong discontinuities in fluids and MHD. In: {\it Handbook of mathematical fluid dynamics}, S. Friedlander, D. Serre
(eds.), {\bf 1}, North-Holland, Amsterdam, 2002, pp. 545–-652

\bibitem{BFKK}
I.B. Bernstein, E.A. Frieman, M.D. Kruskal and R.M. Kulsrud, {\em An energy principle for hydromagnetic stability problems}, {Proc. Roy. Soc. A} {244}  (1958), 17--40

\bibitem{Col}
J.-F. Coulombel, Weakly stable multidimensional shocks. {\it Ann. Inst. H. Poincar\'{e} Anal. Non Lin\'{e}aire}
{\bf 21}  (2004), 401--443

\bibitem{CS2}
J.-F. Coulombel, P. Secchi, Nonlinear compressible vortex sheets in two space dimensions. {\it Ann. Sci. \'Ecole Norm. Sup. (4)} \textbf{41}  (2008), 85--139

\bibitem{Ebin}
D. Ebin, The equations of motion of a perfect fluid with free boundary are not well-posed. {\it Comm. Partial Differential Equations} {\bf 12} (1987), 1175–-1201

\bibitem{Goed}
J.P. Goedbloed,  R. Keppens and S. Poedts, {\em Advanced magnetohydrodynamics: with applications to laboratory and astrophysical plasmas}. Cambridge University Press, Cambridge,  2010

\bibitem{Kreiss} H.-O. Kreiss,  Initial boundary value problems for hyperbolic systems. {\it Commun. Pure Appl. Math.}  {\bf 23} (1970), 277--296

\bibitem{Lannes}
D. Lannes, Well-posedness of the water-waves equations. {\it J. Amer. Math. Soc.} {\bf 18}  (2005), 605--654

\bibitem{Lind_incomp}
H. Lindblad, Well-posedness for the motion of an incompressible liquid with free surface boundary. {\it Ann. of Math.} (2) {\bf 162}  (2005), 109--194

\bibitem{Lind}
H. Lindblad,  Well posedness for the motion of a compressible liquid with free surface boundary. {\it Comm. Math. Phys.} \textbf{260} (2005), 319--392

\bibitem{Maj84}
A. Majda, {\it Compressible Fluid Flow and Systems of Conservation Laws in Several Space Variables}. Springer-Verlag,  New York, 1984

\bibitem{Met}
 G. M\'etivier, Stability of multidimensional shocks. In:
{\it Advances in the theory of shock waves}, H. Freist\"uhler, A. Szepessy (eds.), Progr. Nonlinear Differential Equations Appl. {\bf 47}, Birkh\"auser, Boston, 2001, pp. 25--103

\bibitem{MoTraTre-vacuum}
A. Morando, Y. Trakhinin, P. Trebeschi, Well-posedness of the linearized plasma-vacuum interface problem in ideal incompressible MHD. {\it Quart. Appl. Math.} {\bf 72} (2014), 549--587

\bibitem{MTT-cont}
A. Morando, Y. Trakhinin, P. Trebeschi, Well-posedness of the linearized problem for MHD contact discontinuities. {\it J.  Differential Equations} {\bf 258} (2015), 2531--2571

\bibitem{Rauch}
J. Rauch, Symmetric positive systems with boundary characteristic of constant multiplicity. {\it Trans. Amer. Math.
Soc.} {\bf 291} (1985), 167--187.

\bibitem{ST1}
P. Secchi, Y. Trakhinin,  Well-posedness of the linearized plasma-vacuum interface problem. {\it Interface Free Bound.} {\bf 15}  (2013), 323--357

\bibitem{ST}
P. Secchi, Y. Trakhinin, Well-posedness of the plasma-vacuum interface problem. {\it Non\-linearity} {\bf 27} (2014), 105--169

\bibitem{T05}
Y. Trakhinin, On existence of compressible current-vortex sheets: variable coefficients linear analysis. {\it Arch. Ration. Mech. Anal.} {\bf 177} (2005), 331--366

\bibitem{T09}
Y. Trakhinin, The existence of current-vortex sheets in ideal compressible magnetohydrodynamics. {\it Arch. Ration. Mech. Anal.} {\bf 191} (2009), 245--310

\bibitem{Tcpam}
Y. Trakhinin,  Local existence for the free boundary problem for nonrelativistic and relativistic compressible Euler equations with a vacuum boundary condition. {\it Comm. Pure Appl. Math.} {\bf 62} (2009), 1551--1594

\bibitem{Tjde}
Y. Trakhinin,  On the well-posedness of a linearized plasma-vacuum interface problem in ideal compressible MHD. {\it J. Differential Equations} {\bf 249}  (2010), 2577--2599


\end{thebibliography}
\end{document}